\documentclass[10pt]{article}
\usepackage{amsmath}
\usepackage{amssymb}
\usepackage{amsthm}
\usepackage{graphicx}
\usepackage{relsize}
\usepackage{verbatim}
\usepackage{enumerate}
\usepackage{color}
\usepackage{mathrsfs}
\usepackage{tikz}
\usepackage{subfigure}
\usepackage{longtable}
\numberwithin{equation}{section}

\theoremstyle{plain}
\newtheorem{thm}{Theorem}[section]

\newtheorem{lem}[thm]{Lemma}
\newtheorem{pr}[thm]{Proposition}
\newtheorem{cor}[thm]{Corollary}

\newtheorem{defns}[thm]{Definitions}

\newtheorem{question}{Question}

\theoremstyle{remark}

\newtheorem*{unremark}{Remark}

\newtheorem*{unremarks}{Remarks}

\allowdisplaybreaks

\def\N{\mathbb{N}}
\def\Z{\mathbb{Z}}

\def\CC{{\cal C}}

\def\ee{\varepsilon}
\def\E{{\mathbf E}}
\def\P{{\mathbf P}}
\def\Cox{\hfill \Box}

\def\disp{\displaystyle}
\def\one{{\bf 1}}
\def\cc{{\bf c}}
\def\|{{\, | \, }}
\def\F{{\mathcal F}}
\def\T{{\mathcal T}}
\def\V{{\mathcal V}}
\def\B{{\mathcal B}}

\def\CC{{\mathcal C}}

\def\rt{{\bf 0}}
\def\parent{{\rm par}\,}
\def\ulam{{\mathcal U}}
\def\VU{{\bf V}}  
\def\bfa{{\bf a}}

\def\conn{\leftrightarrow}

\newcommand{\Bon}{\mathsf{Bon}}
\newcommand{\open}{\mathsf{Open}}
\newcommand{\Nobranch}{\mathsf{NoBranch}}
\newcommand{\leafbranch}{\mathsf{LeafBranch}}
\newcommand{\branch}{\mathsf{Branch}}

\newcommand{\subt}{\mathsf{SubTree}}
\def\Phiinv{\iota}
\def\embed{\iota}
\def\G{{\mathcal G}}

\def\ul{\underline}
\def\GW{{\tt GW}}

\def\one{\mathbf{1}}

\def\rtt{{\mathbf{0}}}
\def\TT{{\bf T}}
\def\Trp{{\mathcal R}}
\def\TFV{{\mathcal D}}

\def\rmj{r_{m,j}}
\def\A{{\mathcal A}}
\def\red{\color{red}}

\newcommand{\Ttilde}{\widetilde{T}}

\def\gt{{\tilde{g}}}

\begin{document}
	
\begin{center}
{\large \bf Quenched Survival of Bernoulli Percolation on Galton-Watson Trees}
\end{center}
	
\begin{flushright}
	Marcus Michelen, Robin Pemantle and Josh Rosenberg \\
	{\tt \{marcusmi, pemantle, rjos\}@math.upenn.edu}
\end{flushright}
\begin{abstract}
We explore the survival function for percolation on Galton-Watson trees.  
Letting $g(T,p)$ represent the probability a tree $T$ survives Bernoulli 
percolation with parameter $p$, we establish several results about the 
behavior of the random function $g(\TT , \cdot)$, where $\TT$ is drawn from
the Galton-Watson distribution.  These include almost sure smoothness 
in the supercritical region; an expression for the $k\text{th}$-order 
Taylor expansion of $g(\TT , \cdot)$ at criticality in terms of limits of 
martingales defined from $\TT$ (this requires a moment condition depending 
on $k$); and a proof that the $k\text{th}$ order derivative extends 
continuously to the critical value.  Each of these results is shown 
to hold for almost every Galton-Watson tree.
\end{abstract}
	

\setcounter{section}{0}
\section{Introduction} \label{sec:intro}
	
Let $\GW$ denote the measure on locally finite rooted trees induced 
by the Galton-Walton process for some fixed progeny distribution $\{ p_n \}$
whose mean will be denoted $\mu$.  A random tree generated according
to the measure $\GW$ will be denoted as $\TT$.  Throughout, we let $Z$ 
denote a random variable with distribution $\{ p_n \}$ and assume that 
$\P[Z=0]=0$; passing to the reduced tree as described in 
\cite[Chapter 1.D.12]{athreya-ney}, no generality is lost for any of
the questions in the paper.

The growth rate and regularity properties of both random and deterministic 
trees can be analyzed by looking at the behavior of a number of different 
statistics.  The Hausdorff dimension of the boundary and the escape speed 
of random walk are almost surely constant for a fixed Galton-Watson measure.  
Quantities that are random but almost surely well defined include the 
martingale limit $W := \lim Z_n / \mu^n$, the resistance to infinity when 
edges at level $n$ carry resistance $x^n$ for a fixed $x < \mu$, and the 
probability $g(\TT,p)$ that $\TT$ survives Bernoulli-$p$ percolation, i.e., 
the probability there is a path of open edges from the root to infinity, 
where each edge is declared open with independent probability $p$.  
In this paper we seek to understand $\GW$-almost sure regularity properties 
of the survival function $g(\TT,\cdot)$ and to compute its derivatives
at criticality.
	
The properties of the Bernoulli-$p$ percolation survival function 
have been studied extensively in certain other cases, such as on 
the deterministic $d$-dimensional integer lattice, $\Z^d$.  When $d=2$, 
the Harris-Kesten Theorem~\cite{harris60,kesten-80} states that the 
critical percolation parameter $p_c$ is equal to $1/2$ and that
critical percolation does not survive: $g(\Z^2,1/2) = 0$; more interesting
is the nondifferentiability from the right of the survival function at
criticality~\cite{kesten-zhang}.  When $d\geq 3$, less is known, despite 
the high volume of work on the subject.  The precise value of the 
critical probability $p_c(d)$ is unknown for each $d \geq 3$; for 
$d \geq 19$, mean-field behavior has been shown to hold, implying 
that percolation does not occur at criticality~\cite{hara-slade}.  
This has recently been upgraded with computer assistance and shown to 
hold for $d \geq 11$ \cite{F-vdH}, while the cases of $3 \leq d \leq 10$ 
are still open.  Lower bounds on the survival probability of $\Z^d$ in the 
supercritical region are an area of recent work~\cite{dumin-copin-tassion}, 
but exact behavior near criticality is not known in general.  On the 
question of regularity, the function $g(\Z^d,p)$ is smooth on 
$(p_c(d), 1]$ for each $d\geq 2$ \cite[Theorem $8.92$]{grimmett}.

There is less known about the behavior of $g(\TT , \cdot)$ for random trees
than is known on the integer lattice.  We call the random function 
$g(\TT , \cdot)$ the {\em quenched survival function} to distinguish
it from the {\em annealed survival function} $g$, where $g(x)$ is
the probability of survival at percolation parameter $x$ averaged
over the $\GW$ distribution.
For the regular $d$-ary tree, $T_d$, the classical theory of 
branching processes implies that the critical percolation parameter 
$p_c$ is equal to $1/d$, that $g(T_d , 1/d) = 0$ (that is, 
there is no percolation at criticality), and that for $p > p_c$, 
the quantity $g(T_d , p)$ is equal to the largest fixed point 
of $1 - (1-px)^d$ in $[0,1]$ (see, for instance,~\cite{athreya-ney}
for a treatment of this theory).  

For Galton-Watson trees, a comparison of the quenched and annealed
survival functions begins with the following classical result of Lyons,
showing that $p_c$ is the same in both cases.

\begin{thm}[\cite{lyons90}] \label{th:lyons}
Let $\TT$ be the family tree of a Galton-Watson process with mean 
$\E[Z] =: \mu > 1$, and let $p_c (\TT) = \sup \left \{p \in [0,1] : 
g(\TT,p) = 0\right\}$.  Then $p_c(\TT)=\frac{1}{\mu}$ almost surely.
Together with the fact that $g(\TT , 1/\mu) = 0$, 
this implies $g(\TT,p_c) = 0$ almost surely. 
$\Cox$
\end{thm}

To dig deeper into this comparison, observe first that
the annealed survival probability $g(x)$ is the unique fixed point on
$[0,1)$ of the function $1 - \phi(1 - px)$ where $\phi (z) = \E z^Z$ 
is the probability generating function of the offspring distribution.   
In the next section we show that the annealed survival function $g(p)$
is smooth on $(p_c , 1)$ and, under moment conditions on the offspring 
distribution, the derivatives extend continuously to $p_c$.
This motivates us to ask whether the same holds for the quenched 
survival function.  Our main results show this to be the case, 
giving regularity properties of $g(\TT,p)$ on the supercritical region.  

Let $r_j$ be the coefficents in the asymptotic expansion of
the annealed function $g$ at $p_c$.  These are shown to exist in 
Proposition~\ref{pr:g-rec-rel} below. In Theorem~\ref{th:g-expansion}, 
under appropriate moment conditions, we will construct for each 
$j \geq 1$ a martingale $\{M^{(j)}_n : n \geq 1 \}$ with an almost 
sure limit $M^{(j)}$, that is later proven to equal the $j$th coefficient 
in the aymptotic expansion of the quenched survival function $g$ at $p_c$.
Throughout the analysis, the expression $W$ denotes the martingale limit
$\text{lim}Z_n/\mu^n$.

\begin{thm}[main results] \label{th:main}
~~\\[-2ex]
\begin{enumerate}[(i)]
\item For $\GW$ a.e. tree $\TT$, the quantity $g(\TT,x)$ is smooth 
as a function of $x$ on $(p_c,1)$.  
\item If $\E Z^{2k+1+\beta} < \infty$ for some positive integer $k$ and 
some $\beta > 0$, then we have the $k$-th order approximation 
$$g(\TT,p_c + \ee) = \sum_{j = 1}^k M^{(j)} \ee^j+ o(\ee^k)$$ 
for $\GW$ a.e. tree $\TT$, where $M^{(j)}$ is the quantity given 
explicitly in Theorem~\ref{th:g-expansion}.  Additionally, 
$M^{(1)} = W r_1$ and $\E[M^{(j)}] = r_j$, where $W$ is the martingale
limit for $\TT$ and $j! r_j$ are the derivatives of the annealed
survival function, for which explicit expressions are given in
Proposition~\ref{pr:g-rec-rel}.
\item If $\E Z^{2k^2 + 3 + \beta} < \infty$ for some $\beta>0$, then 
$\GW$-almost surely $g(\TT , \cdot)$ is of class $C^k$ from the right 
at $p_c$ and $g^{(j)} (\TT , p_c^+) = j! M^{(j)}$ for all $j\leq k$; 
see the beginning of Section~\ref{ss:smooth} for calculus definitions.
\end{enumerate}
\end{thm}

\begin{unremarks}
Smoothness of $g(\TT,\cdot)$ on $(p_c,1)$ does not require any 
moment assumptions, in fact even when $\E Z = \infty$ one has 
$p_c = 0$ and smoothness of $g(\TT,\cdot)$ on $(0,1)$.  The
moment conditions relating to expansion at criticality given in (ii) 
are probably not best possible, but are 
necessary in the sense that if $\E Z^k = \infty$ for some $k$
then not even the annealed survival function is smooth
(see Proposition~\ref{pr:nonsmpr} below).
\end{unremarks}
	
The proofs of the first two parts of Theorem~\ref{th:main} are independent
of each other.  Part~$(ii)$ is proved first, in Section~\ref{sec:critical}.
Part~$(i)$ is proved in Section~\ref{sec:c3} after some preliminary 
work in Section~\ref{sec:g'}.  Finally, part~$(iii)$ is proved in 
Section~\ref{sec:cont-at-pc}.  
The key to these results lies in a number of different expressions
for the probability of a tree $T$ surviving $p$-percolation and
for the derivatives of this with respect to $p$.  The first of these expressions 
is obtained via inclusion-exclusion.  The second, Theorem~\ref{th:c1} below,
is a Russo-type formula~\cite{russo} expressing the derivative in terms 
of the {\em expected branching depth}
$$\frac{d}{dp} g(T,p) = \frac{1}{p} \E_T |B_p| \, $$
for $\GW$-almost every $T$ and every $p \in (p_c , 1)$, where
$|B_p|$ is the depth of the deepest vertex $B_p$ whose removal
disconnects the root from infinity in $p$-percolation.
The third generalizes this to a combinatorial construction suitable 
for computing higher moments.  

A brief outline of the paper is as follows.  Section~\ref{sec:prelim}
contains definitions, preliminary results on the annealed 
survival function, and a calculus lemma.  Section~\ref{sec:critical} 
writes the event of survival to depth $n$ as a union over the events of 
survival of individual vertices, then obtains bounds via inclusion-exclusion.  
Let $X_n^{(j)}$ denote the expected number of cardinality $j$ sets
of surviving vertices at level $n$, and let $X_n^{(j,k)}$ 
denote the expected $k$th falling factorial of this quantity.
These quantities diverge as $n \to \infty$ but inclusion-exclusion
requires only that certain signed sums converge as $n \to \infty$.  
The Bonferroni inequalities give upper and lower bounds on 
$g(T,\cdot)$ for each $n$.  Strategically choosing $n$ as a function 
of $\ee$ and using a modified Strong Law argument allows us to ignore 
all information at height beyond $n$ (Proposition \ref{pr:bonf-replace}).  
Each term in the Bonferroni inequalities is then individually 
Taylor expanded, yielding an expansion of $g(T,p_c + \ee)$ with 
coefficients depending on $n$.  
Letting $\TT \sim \GW$ and $n \to \infty$, the variables $X_n^{(j,k)}$
separate into a martingale part and a combinatorial part.  The martingale 
part converges exponentially rapidly (Theorem~\ref{th:Y-jk-mart}).  
The martingale property for the cofficients themselves  
(Lemma~\ref{lem:M_n-mart}) follows from some further 
analysis (Lemma~\ref{lem:constants-sum}) eliminating the 
combinatorial part when the correct signed sum is taken.

Section~\ref{sec:g'} proves the above formula for the derivative of $g$ 
(Theorem~\ref{th:c1}) via a Markov
property for the coupled percolations as a function of the
percolation parameter $p$.  Section~\ref{sec:c3} begins with 
a well-known branching process description of the subtree of
vertices with infinite lines of descent.  It then goes on to describe
higher order derivatives in terms of combinatorial gadgets denoted 
$\TFV$ which are moments of the numbers of edges in certain
rooted subtrees of the percolation cluster and generalize the
branching depth.  We then prove an identity for differentiating these
(Lemma~\ref{lem:iterated}), and apply it repeatedly to $g' (T,p) = 
p^{-1} \E B_p$,
to write $(\partial / \partial p)^k g (T , p)$ as a finite sum
$\sum_\alpha \TFV_\alpha$ of factorial moments of sets of surviving
vertices.  This suffices to prove smoothness of the quenched survival 
function on the supercritical region $p_c < p < 1$.

For continuity of the derivatives at $p_c$, an analytic trick
is required.  If a function possessing an order $N$ asymptotic expansion 
at the left endpoint of an interval $[a,b]$ ceases to be of class
$C^k$ at the left endpoint for some $k$, then the $k+1$st derivative
must blow up faster than $(x-a)^{-N/k}$ (Lemma~\ref{lem:g-upper-bound}).
This is combined with bounds on how badly things can blow up at $p_c$ 
(Proposition~\ref{pr:k+1/k}) to prove continuity from the right at $p_c$ 
of higher order derivatives. 

The paper ends by listing some questions left open, concerning sharp
moment conditions and whether an asymptotic expansion ever 
exists without higher order derivatives converging at $p_c$.

\section{Constructions, preliminary results, and annealed survival} 
\label{sec:prelim} 

\subsection{Smoothness of real functions at the left endpoint}
\label{ss:smooth}

Considerable work is required to strengthen
conclusion~$(ii)$ of Theorem~\ref{th:main} to conclusion~$(iii)$.
For this reason, we devote a brief subsection to making the calculus 
statements in Theorem~\ref{th:main} precise.
We begin by recalling some basic facts about one-sided derivatives
at endpoints, then state a useful lemma.

Let $f$ be a smooth function on the nonempty real interval $(a,b)$.
We say that $f$ is of class $C^k$ from the right at $a$ if $f$ and 
its first $k$ derivatives extend continuously to some finite values 
$c_0, \ldots , c_k$ when approaching $a$ from the right.  
This implies, for each $j \leq k$, that 
$\disp \ee^{-1} \left ( f^{(j-1)} (a+\ee) - f^{(j-1)} (a) \right ) 
\to c_j$ as $\ee \downarrow 0$, and therefore that $c_j$ is the $j$th
derivative of $f$ from the right at $a$.  It follows from Taylor's
theorem that there is an expansion with the same coefficients:
\begin{equation} \label{eq:taylor}
f(x) = f(a) + \sum_{j=1}^k \frac{c_j}{j!} (x-a)^j + o(x-a)^k
\end{equation}
as $x \to a^+$.

The converse is not true: the classical example $f(x) = x^k \sin (1/x^m)$
shows that it is possible to have the expansion~\eqref{eq:taylor} 
with $f$ ceasing to be of class $C^j$ once $j(m+1) \geq k$.  On the
other hand, the ways in which the converse can fail are limited.  We
show that if $f$ possesses an order-$k$ expansion as in~\eqref{eq:taylor} 
and if $j \leq n$ is the least integer for which $f \notin C^j$, then 
$f^{(j+1)}$ must oscillate with an amplitude that blows up at least 
like a prescribed power of $(x-a)^{-1}$.  

\begin{lem} \label{lem:g-upper-bound}
Let $f:[a,b]\to\mathbb{R}$ be $C^{\infty}$ on $(a,b)$ with 
\begin{equation}\label{exfordetf}
f(a+\ee) = c_1\ee + \dots + c_k\ee^k + \dots + c_N \ee^N + o(\ee^N)
\end{equation}
for some $k,N$ with $1\leq k < N$, and assume 
\begin{equation} \label{eq:j<k}
\lim_{\ee\to 0} f^{(j)}(a+\ee) = j! c_j
\end{equation}
for all $j$ such that $1 \leq j < k$.  If $f^{(k)}(a+\ee)
\not\to k! c_k$ as $\ee \to 0^+$, then there must exist positive
numbers $u_n \downarrow 0$ such that
$$\left | f^{(k+1)} (u_n) \right | = \omega \Big( u_n^{-\frac{N}{k}} 
   \Big) \, .$$
$\Cox$
\end{lem}

\noindent{\sc Proof:} ~~\\[2ex]
\ul{Step 1:}
Assume without loss of generality that $a=0$.  Also, replacing $f$ by
$f - q$ where $q$ is the polynomial $q(x) := \sum_{j=1}^N c_j x^j$,
we may assume without loss of generality that $c_j = 0$ for $j \leq N$.

\ul{Step 2:}
Fixing $f$ satisfying~\eqref{exfordetf}--\eqref{eq:j<k}, we claim that  
\begin{equation} \label{eq:liminf}
\liminf_{x \downarrow 0} f^{(k)} (x) \leq 0 \, .
\end{equation}
To see this, assume to the contrary and choose $c, \delta > 0$ with
$f^{(k)}(x) \geq c$ on $(0,\delta)$.  Using $f^{(k-1)} (0) = 0$ and 
integrating gives
$$f^{(k-1)} (x) = \int_0^x f^{(k)} (t) \, dt \geq c x$$
for $0 < x < \delta$.  Repeating this argument and using induction, 
we see that $f^{(k-j)} (x) \geq c x^j / j!$, whence $f(x) \geq c x^k / k!$
on $(0,\delta)$ contradicting $f(x) = o(x^N)$.

\ul{Step 3:}
An identical argument shows that $\limsup_{x \downarrow 0} f^{(k)} (x) 
\geq 0$.  Therefore,
$$\liminf_{x \downarrow 0} f^{(k)} (x) \leq 0 \leq 
   \limsup_{x \downarrow 0} f^{(k)} (x) \, .$$
Assuming as well that $f^{(k)} (x) \not \to 0$ as $x \downarrow 0$, 
at least one of these inequalities must be strict.  Without loss of
generality, we assume for the remainder of the proof that it is the
second one.  Because $f^{(k)}$ is continuous, we may fix $C > 0$ and
intervals with both endpoints tending to zero such that $f^{(k)}$ is 
equal to $C$ at one endpoint, $C/2$ at the other endpoint, and is 
at least $C/2$ everywhere on the interval.

\ul{Step 4:}
Let $J = [a,b]$ (not to be confused with the $[a,b]$ from the statement of the lemma) denote one of these intervals.  We next claim that we
can find nested subintervals $J = J_0 \supseteq J_1 \supseteq \cdots \supseteq 
J_k$ with lengths $|J_j| = 4^{-j} |J_0|$ such that 
$$|f^{(k-j)}| > \frac{C}{2 \cdot 4^{j(j+1)/2}} (b-a)^j \mbox{ on } J_j \, .$$
This is easily seen by induction on $j$.  The base case $j=0$ is already 
done.  Assume for induction we have chosen $J_j$.  The function
$f^{(k-j)}$ does not change sign on $J_j$, hence $f^{(k-j-1)}$ is 
monotone on $J_j$.  If $f^{(k-j-1)}$ has a zero in the first half 
of $J_j$, we let $J_{j+1}$ be the last $1/4$ of $J_j$ and let $I$
denote the second $1/2$ of $J_j$.  If $f^{(k-j-1)}$ has a zero in 
the second half of $J_j$, we let $J_{j+1}$ be the first $1/4$ of $J_j$
and $I$ denote the first half of $J_j$.  If $f^{(k-j-1)}$ does not vanish
on $J_j$ we let $J_{j+1}$ be the first or last $1/4$ of $J_j$, whichever
contains the endpoint at which the absolute value of the monotone function
$f^{(k-j-1)}$ is maximized over $J_j$, letting $I$ denote the half of 
$J_j$ containing $J_{j+1}$.  Figure~\ref{fig:intervals} shows an example 
of the case where $f^{(k-j-1)}$ has a zero in the second half of $J_j$.

In all of these cases, $f^{(k-j-1)}$
is monotone and does not change sign on $I$, and its minimum modulus
on $J_{j+1}$ is at least the length of $I \setminus J_{j+1}$ times
the minimum modulus of $f^{(k-j)}$ on $J_j$.  By induction, this is
at least $4^{-(j+1)} (b-a) \times C (b-a)^j / (2 \cdot 4^{j(j+1)/2})$,
which is equal to $C (b-a)^{j+1} / (2 \cdot 4^{(j+1)(j+2)/2})$, establishing
the claim by induction.

\ul{Step 5:} 
Let $[a_n , b_n]$ be intervals tending to zero as in Step~3 and let $\xi_n$ 
denote the right endpoint of the $(k+1)$th nested sub-interval $J_k$ of $[a_n , 
b_n]$ as in Step~4.  Denoting $C' := C / (2 \cdot 4^{k (k+1) / 2})$, 
we have $|f(\xi_n)| \geq C' (b_n-a_n)^k$; together with  
the hypothesis that $f(x) = o(x^N)$, this implies $(b_n-a_n)^k = o(\xi_n^N)$ 
as $n \to \infty$.  Seeing that $b_n - a_n = o(\xi_n)$ for some $\xi_n
\in [a_n , b_n]$ is enough to conclude that $a_n \sim b_n \sim \xi_n$.
By the Mean Value Theorem, because $f^{(k)}$ transits from $C/2$ to $C$
on $[a_n , b_n]$, there is some point $u_n \in [a_n , b_n]$ with
$$|f^{(k+1)} (u_n)| \geq \frac{C/2}{b_n - a_n} \, .$$
We have seen that $(b_n - a_n)^{-1} = \omega (\xi_n^{-N/k})$.  Because
$\xi_n \sim a_n \sim b_n \sim u_n$, this implies 
$|f^{(k+1)} (u_n)| = \omega (u_n^{-N/k})$, proving the lemma.
	\begin{figure}[!ht]
	\centering
	\includegraphics[height=2.5in]{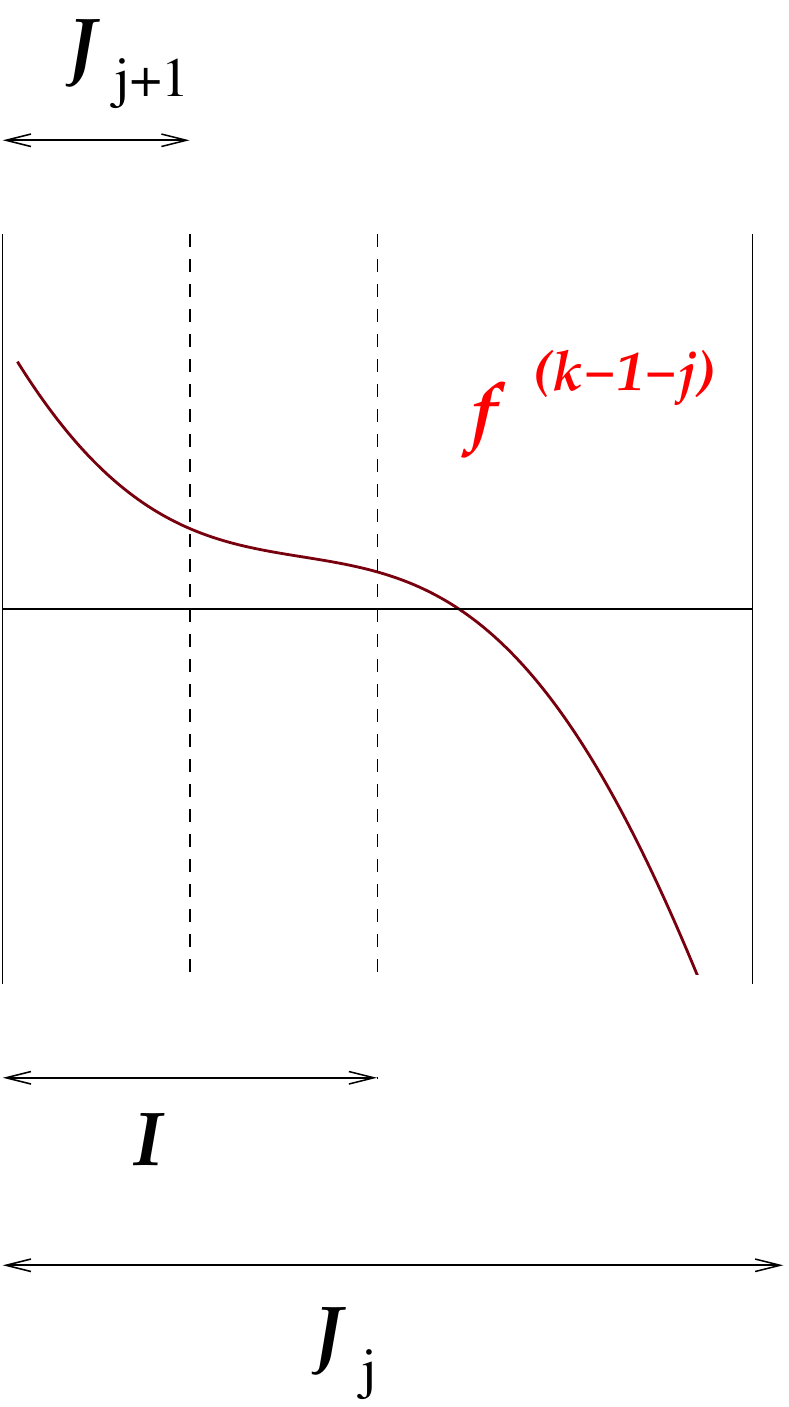}
	\caption{One of four possible cases for the choice of the nested 
	sub-interval $J_{j+1}$}
	\label{fig:intervals}
	\end{figure}
$\Cox$

\subsection{Galton-Watson trees}
	
Since we will be working with probabilities on random trees, it will 
be useful to explicitly describe our probability space and notation. 
We begin with some notation we use for all trees, random or not.
Let $\ulam$ be the canonical {\em Ulam-Harris} tree~\cite{louigi-ford}.
The vertex set of $\ulam$ is the set $\VU : = \bigcup_{n=1}^\infty \N^n$, 
with the empty sequence $\rt = \emptyset$ as the root.  There is an 
edge from any sequence $\bfa = (a_1, \ldots , a_n)$ to any extension 
$\bfa \sqcup j := (a_1, \ldots , a_n , j)$.  The depth of a vertex $v$
is the graph distance between $v$ and $\rt$ and is denoted $|v|$.
We work with trees $T$ that are locally finite rooted subtrees of $\ulam$.
The usual notations are in force: $T_n$ denotes the set of vertices 
at depth $n$; $T(v)$ is the subtree of $T$ at $v$, canonically identified
with a rooted subtree of $\ulam$, in other words the vertex set of 
$T(v)$ is $\{ w : v \sqcup w \in V(T) \}$ and the least common ancestor of 
$v$ and $w$ is denoted $v \wedge w$.  
	
Turning now to Galton-Watson trees, let $\phi (z) := 
\sum_{n=1}^\infty p_n z^n$ be the offspring generating function 
for a supercritical branching process with no death, i.e., $\phi (0) = 0$.  
We recall,
\begin{eqnarray*}
	\phi'  (1) & = & \E Z =: \mu \\ 
	\phi'' (1) & = & \E [Z (Z-1)] 
\end{eqnarray*}
where $Z$ is a random variable with probability generating function $\phi$.  
We will work on the canonical probability space $(\Omega , \F , \P)$
where $\Omega = (\N \times [0,1])^{\VU}$ and $\F$ is the product Borel
$\sigma$-field.  We take $\P$ to be the probability measure making the
coordinate functions $\omega_v = (\deg_v , U_v)$ i.i.d. with the law
of $(Z,U)$, where $U$ is uniform on $[0,1]$ and independent of $Z$. 
The variables $\{ \deg_v \}$, where $\deg_v$ is interpreted as the number of
children of vertex $v$, will construct the Galton-Watson tree, 
while the variables $\{ U_v \}$ will be used later for percolation.  
Let $\TT$ be the random rooted subtree of $\ulam$ which is the 
connected component containing the root of the set of vertices 
that are either the root or are of the form $v \sqcup j$ such that 
$0 \leq j < \deg_v$.  This is a Galton-Watson tree with offspring
generating function $\phi$.  Let $\T:=\sigma( \left\{ \deg_v \right\})$ 
denote the $\sigma$-field generated by the tree $\TT$.  The $\P$-law
of $\TT$ on $\T$ is $\GW$.
	
As is usual for Galton-Watson branching processes, we denote
$Z_n := |\TT_n|$.  Extend this by letting $Z_n (v)$ denote the
number of offspring of $v$ in generation $|v| + n$; similarly, extend the
notation for the usual martingale $W_n := \mu^{-n} Z_n$ by letting
$W_n (v) := \mu^{-n} Z_n (v)$.  We know that $W_n (v) \to W(v)$
for all $v$, almost surely and in $L^q$ if the offspring distribution
has $q$ moments.   This is stated without
proof for integer values of $q \geq 2$ in~\cite[p.~16]{harris-BP} 
and \cite[p.~33, Remark~3]{athreya-ney}; for a proof for all $q > 1$,
see~\cite[Theorems~0~and~5]{bingham-doney74}.  Further extend this 
notation by letting $v^{(i)}$
denote the $i$th child of $v$, letting $Z_n^{(i)} (v)$ denote
$n$th generation descendants of $v$ whose ancestral line passes
through $v^{(i)}$, and letting $W_n^{(i)} (v) := \mu^{-n} Z_n^{(i)} (v)$.
Thus, for every $v$, $W(v) = \sum_i W^{(i)} (v)$.  For convenience, 
we define $p_c := 1/\mu$, and recall that $p_c$ is in fact $\GW$-a.s. 
the critical percolation parameter of $T$ as per Theorem \ref{th:lyons}.

\subsubsection*{Bernoulli percolation}\label{sec:bpintro}
	
Next, we give the formal construction of Bernoulli percolation
on random trees.  For $0 < p < 1$, simultaneously define Bernoulli$(p)$ 
percolations on rooted subtrees $T$ of $\ulam$ by taking the percolation 
clusters to be the connected component containing $\rtt$ of the induced 
subtrees of $T$ on all vertices $v$ such that $U_v \leq p$.
Let $\F_n$ be the $\sigma$-field generated by the variables 
$\{ U_v , \deg_v : |v| < n \}$.  Because percolation is often imagined
to take place on the edges rather than vertices, we let $U_e$ be a synonym
for $U_v$, where $v$ is the farther of the two endpoints of $e$ from the root.
Write $v \conn_{T,p} w$ if 
$U_e \leq p$ for all edges $e$ on the geodesic from $v$ to $w$ in $T$.  
Informally, $v \conn_{T,p} w$ iff $v$ and $w$ are both in $T$ and 
are connected in the $p$-percolation.  The event of successful $p$-percolation 
on a fixed tree $T$ is denoted 
$H_T (p) := \{ \rtt \conn_{T,p} \infty \}$.  The event of 
successful $p$-percolation on the random tree $\TT$,
is denoted $H_\TT (p)$ or simply $H(p)$.  Let $g(T , p) := 
\P [H_T (p)]$ denote the probability of $p$-percolation
on the fixed tree $T$.  Evaluating at $T = \TT$ gives the random
variable $g(\TT , p)$ which is easily seen to equal the conditional
expectation $\P (H(p) \| \T)$.  Taking unconditional expectations
we see that $g(p) = \E g(\TT , p)$.  

\subsection{Smoothness of the annealed survival function $g$}
	
By Lyons' theorem, $g(p_c) = \E g(\TT , p_c) = 0$.  We now record some 
further properties of the annealed survival function $g$.  
	
\begin{pr} \label{pr:K}
The derivative from the right $K := \partial_+ g(p_c)$ exists and is given by 
\begin{equation} \label{eq:K}
K = \frac{2}{p_c^3 \phi'' (1)} \, .
\end{equation}
where $1/\phi''(1)$ is interpreted as $\lim_{\xi \to 1^-} 1/\phi''(\xi)$.
\end{pr} 	
	
\begin{proof} Let $\phi_p (z) := \phi (1 - p
+ p z)$ be the offspring generating function for the 
Galton-Watson tree thinned by $p$-percolation for $p \in (p_c,1)$.  
The fixed point of $\phi_p$ is $1 - g(p)$.  In other words, $g(p)$ is the
unique $s \in (0,1)$ for which $1 - \phi_p (1-s) = s$, i.e. 
$1 - \phi(1 - p s) = s$.  By Taylor's theorem with Mean-Value remainder, 
there exists a $\xi \in (1 - pg(p),1)$ so that 
$$1 - \phi (1-pg(p)) = pg(p) \phi'(1) - \frac{p^2g(p)^2}{2}\phi''(\xi) 
   = \frac{p}{p_c}g(p) - \frac{p^2 g(p)^2}{2} \phi''(\xi)\,.$$
Setting this equal to $g(p)$ and solving yields 
$$\frac{g(p)}{p - p_c} = \frac{2}{p_c p^2 \phi''(\xi)} \, .$$ 
Taking $p \downarrow p_c$ and noting $\xi \to 1$ completes the proof.
$\Cox$
\end{proof}
	
\begin{cor}\label{cor:analytic} 
$(i)$ The function $g$ is analytic on $(p_c,1)$.  
$(ii)$ If $p_n$ decays exponentially then $g$ is analytic on
$[p_c,1)$, meaning that for some $\ee > 0$ there is an analytic 
function $\gt$ on $(p_c - \ee , 1)$ such that $g(p) = \gt (p) \one_{p > p_c}$.
\end{cor}
	
\begin{proof}
Recall that for $p \in (p_c,1)$, $g(p)$ is the unique positive $s$  
that satisfies $s= 1- \phi(1 - ps)$.  It follows that for all 
$p\in(p_c,1)$, $g(p)$ is the unique $s$ satisfying 
$$F(p,s) := s + \phi(1 - ps) - 1 = 0 \, .$$ 
Also note that since $\phi(1 - ps)$ is analytic with respect to 
both variables for $(p,s)\in(p_c,1)\times(0,1)$, this means $F$ is as well.
		
We aim to use the implicit function theorem to show that we can 
parameterize $s$ as an analytic function of $p$ on $(p_c, 1)$; 
we thus must show $\frac{ \partial F}{\partial s} \neq 0$ at 
all points $(p,g(p))$ for $p \in (p_c,1)$.  Direct calculation gives 
$$\frac{\partial F}{\partial s} = 1 - p \phi'(1 - ps) \, . $$  
Because $\phi$ is strictly convex on $(p_c,1)$, we see that 
$\frac{\partial F}{\partial s}$ is positive for $p \in (p_c,1)$ 
at the fixed point.  Therefore, $g(p)$ is analytic on $(p_c,1)$.  

To prove~$(ii)$, observe that $\phi$ extends analytically to a segment
$[0,1+\ee]$, which implies that $1 - \phi (1-ps)$ is analytic on 
a real neighborhood of zero.  Also $1 - \phi (1-ps)$ vanishes at
$s=0$, therefore $\psi (p,s) := (1 - \phi(1-ps)) / s$ is analytic 
near zero and for $(p,s) \in (p_c,1) \times (0,1)$, the least positive
value of $s$ satisfying $\psi (p,s) = 1$ yields $g(p)$.  Observe that
$$\frac{\partial \psi}{\partial p} (p_c,0) = \lim_{s \to 0}
   \frac{s \phi' (1 - p_c s)}{s} = \phi'(1) = \mu \, .$$
By implicit differentiation,
$$\partial_+ g(p_c) 
   = - \frac{\partial \psi / \partial s}{\partial \psi / \partial p} (p_c,0)$$
which is equal to $1/K$ by Proposition~\ref{pr:K}.  In particular, 
$(\partial \psi / \partial s) (p_c,0) = -\mu / K$ is nonvanishing.
Therefore, by the analytic implicit function theorem, solving 
$\psi (p,s) = 1$ for $s$ defines an analytic function $\gt$ taking 
a neighborhood of $p_c$ to a neighborhood of zero, with $\gt (p) > 0$
if and only if $p > p_c$.  We have seen that $\gt$ agrees with $g$
to the right of $p_c$, proving~$(ii)$.
$\Cox$
\end{proof}

In contrast to the above scenario in which $Z$ has exponential moments
and $g$ is analytic at $p_c^+$, the function $g$ fails to be smooth at
$p_c^+$ when $Z$ does not have all moments.
The next two results quantify this: no $k$th moment implies $g \notin
C^k$ from the right at $p_c$, and conversely, $\E Z^k < \infty$ implies
$g \in C^j$ from the right at $p_c$ for all $j < k/2$.  

\begin{pr}\label{pr:nonsmpr} Assume $k\geq 2$, $\E[Z^k]<\infty$, 
and $\E[Z^{k+1}]=\infty$.  Then $g^{(k+1)}(p)$ does not extend 
continuously to $p_c$ from the right.
\end{pr}

\begin{proof}
If $\underset{p\to p_c^+}{\text{lim}}g^{(j)}(p)$ does not exist for any $j\leq k$, then we are done.  Hence, we will assume all such limits exist.  Next we take the $(k+1)$th derivative of each side of the expression $1-\phi(1-pg(p))=g(p)$ in order to get an equality of the form \begin{align}g^{(k+1)}(p) &= \sum_{j=1}^N p^{a_j}\cdot g^{(b_{j,1})}(p)^{c_{j,1}}\dots g^{(b_{j,\ell(j)})}(p)^{c_{j,\ell(j)}}\cdot\phi^{(d_j)}(1-pg(p))\nonumber \\
&\qquad+(-1)^k\Big(g(p)+pg'(p)\Big)^{k+1}\phi^{(k+1)}(1-pg(p))+p\phi'(1-pg(p))g^{(k+1)}(p) \label{kp1derives}\end{align}
where the $b_{j,i}$'s and $d_j$'s are all less than or equal to $k$.  Denoting the sum in \eqref{kp1derives} as $S_{k+1}(p)$ and solving for $g^{(k+1)}(p)$, we now get $$g^{(k+1)}(p)=\frac{S_{k+1}(p)+(-1)^k\Big(g(p)+pg'(p)\Big)^{k+1}\phi^{(k+1)}(1-pg(p))}{1-p\phi'(1-pg(p))}.$$Since we're assuming that $g^{(j)}(p)=O(1)$ as $p\downarrow p_c$ for all $j\leq k$, and because $\E[Z^j]<\infty\implies\phi^{(j)}(1-pg(p))=O(1)$ as $p\downarrow p_c$, it follows that $S_{k+1}(p)=O(1)$.  Combining this with the fact that $\E[Z^{k+1}]=\infty\implies\phi^{(k+1)}(1-pg(p))\to\infty$ as $p\downarrow p_c$, and that $g(p)+pg'(p)\to p_c K>0$ as $p\downarrow p_c$, we now see that the numerator in the above expression for $g^{(k+1)}(p)$ must go to infinity as $p\downarrow p_c$, which means $g^{(k+1)}(p)$ must as well.
$\Cox$
\end{proof}

In Section~\ref{ss:asymp} we will prove the following partial converse.

\begin{pr}\label{pr:simfannc}
For each $k\geq 1$, if $\E[Z^{2k+1}]<\infty$, then $g \in C^k$ from 
the right at $p_c$.
\end{pr}

\subsection{Expansion of the annealed survival function $g$ at $p_c^+$}
\label{ss:asymp}

A good part of the quenched analysis requires only the expansion of 
the annealed survival function $g$ at $p_c^+$, not continuous derivatives.  
Proposition~\ref{pr:g-rec-rel} below shows that $k+1$ moments are enough 
to give the order $k$ expansion.  Moreover, we give explicit expressions
for the coefficients.  We require the following combinatorial construction:  
let $\mathcal{C}_j(k)$ denote the set of compositions of $k$ into 
$j$ parts, i.e. ordered $j$-tuples of positive integers 
$(a_1,\ldots,a_j)$ with $a_1 + \cdots + a_j = k$; for a composition 
$a = (a_1,\ldots,a_j)$, define $\ell(a) = j$ to be the \emph{length} 
of $a$, and $|a| = a_1 + \cdots + a_j$ to be the \emph{weight} of $a$.  
Let $\mathcal{C}(\leq k)$ denote the set of compositions with weight 
at most $k$. 

\begin{pr} \label{pr:g-rec-rel}
Suppose $\E[Z^{k+1}] < \infty$.  Then there exist constants 
$r_1,\dots ,r_k$ such that $g(p_c+\ee)=r_1\ee+\cdots +r_k\ee^k+o(\ee^{k})$.  
Moreover, the $r_j$'s are defined recursively via
\begin{eqnarray}
r_1 & = & g'(p_c) = \frac{2}{p_c^3 \phi''(1)} \, ; \nonumber \\[2ex]
r_j & = & \frac{2}{p_c^2 \phi''(1)} 
   \sum_{\substack{a \in \mathcal{C}(\leq j) \\ a \neq (j) }}	 
   r_{a_1}\cdots r_{a_{\ell(a)}} 
   \binom{\ell(a)+1}{j - |a|}p_c^{|a| + \ell(a)+1 - j}(-1)^{\ell(a)} 
   \frac{\phi^{(\ell(a))+1}(1)}{(\ell(a)+1)!} \, . \label{eq:r}
\end{eqnarray}
\end{pr}

\begin{proof}
To start, we utilize the identity $1- \phi(1 - pg(p)) = g(p)$ for 
$p = p_c + \ee$, and take a Taylor expansion: 
$$\sum_{j = 1}^{k+1} (p_c + \ee)^j g(p_c + \ee)^j 
   (-1)^{j-1}\frac{\phi^{(j)}(1)}{j!} 
   + o(((p_c + \ee)g(p_c + \ee))^{k+1}) = g(p_c + \ee)\,.$$
Divide both sides by $g(p_c + \ee)$ and bound $g(p_c + \ee) = O(\ee)$ 
to get 
\begin{equation}\label{josh55}\sum_{j = 1}^{k+1} (p_c + \ee)^j 
   g(p_c + \ee)^{j-1} (-1)^{j-1}\frac{\phi^{(j)}(1)}{j!} - 1 = o(\ee^k) \, .
\end{equation}
Proceeding by induction, if we assume that the proposition holds for 
all $j<k$ for some $k\geq 2$, and we set 
$$p_k(\ee):=\frac{g(p_c+\ee)-\sum_{j=1}^{k-1}r_j\ee^j}{\ee^k} \, ,$$ 
then \eqref{josh55} gives us
\begin{align}
o(\ee^k) 
&= \sum_{j = 1}^{k+1}(p_c + \ee)^j g(p_c + \ee)^{j-1}(-1)^{j-1}
   \frac{\phi^{(j)}(1)}{j!} - 1\nonumber \\
&= \sum_{j=1}^{k+1}(p_c + \ee)^j \left( \sum_{i = 1}^{k-1}r_i \ee^i  
   + p_k(\ee)\ee^k\right)^{j-1}(-1)^{j-1}\frac{\phi^{(j)}(1)}{j!} - 1 \, . 
   \label{eq:g-deriv-expansion}
\end{align}     
Noting that the assumption that the proposition holds for $j=k-1$ 
implies that $p_k(\ee)=o(\ee^{-1})$, we find that the expression 
on the right hand side in~\eqref{eq:g-deriv-expansion} is the sum of a 
polynomial in $\ee$, the value $\frac{-p_c^2\phi''(1)}{2}p_k(\ee)\ee^k$, 
and an error term which is $o(\ee^{k})$.  This implies that all terms 
of this polynomial that are of degree less than $k$ must cancel, 
and that the sum of the term of order $k$ and 
$\frac{-p_c^2\phi''(1)}{2}p_k(\ee)\ee^k$ must be $o(\ee^k)$.  This 
leaves only terms of degree greater than $k$.  It follows 
that $p_k(\ee)$ must be equal to $C+o(1)$, for some constant $C$.  
		
To complete the induction step, it remains to show that $C=r_k$.  
To do so we must find the coefficient of $\ee^k$ in each term.  
We use the notation $[\ee^j]$ to denote the coefficient of $\ee^j$.
For any $j$, we calculate 
\begin{align*}
[\ee^k] ~(p_c + \ee)^j \left(\sum_{i = 1}^{k-1} r_i \ee^i \right)^{j-1} 
   &= \sum_{r = 1}^k \left( [\ee^r]~\left(\sum_{i = 1}^{k-1} 
      r_i \ee^i \right)^{j-1}  \right)
      \left([\ee^{k - r}]~(p_c + \ee)^{j}\right) \\ 
   &= \sum_{r = 1}^k \left(\sum_{a \in \mathcal{C}_{j-1}(r)} 
      r_{a_1}\cdots r_{a_{j-1}}\right) \binom{j}{k-r} p_c^{j + r - k} \,.
\end{align*}
Putting this together with~\eqref{eq:g-deriv-expansion} we now obtain 
the desired equality $C=r_k$.  Finally, noting that the base case 
$k=1$ follows from Proposition~\ref{pr:K}, we see that the proposition 
now follows by induction. 
$\Cox$
\end{proof}

\noindent{\sc Proof of Proposition}~\ref{pr:simfannc}:  Induct on $k$.
For the base case $k=1$, differentiate both sides of the expression 
$1 - \phi(1-pg(p)) = g(p)$ and solve for $g'(p)$ to get 
$$g'(p)=\frac{g(p)\phi'(1-pg(p))}{1-p\phi'(1-pg(p))}.$$
The numerator and denominator converge to zero as $p \downarrow p_c$.
We would like to apply L'H{\^o}pital's rule but we have to be 
careful because all we have {\em a priori} is the expansion at $p_c$,
not continuous differentiability.  We verify that the denominator satisfies
\begin{equation} \label{eq:order}
1 - p \phi' (1 - p g(p)) \sim \mu \cdot (p - p_c)
\end{equation}
as $p = p_c + \ee \downarrow p_c$ by writing the denominator as  
$1 - (p_c + \ee) (\mu - \phi''(1) p_c r_1 \ee + o(\ee)) = 
\ee (\mu - p_c^2 r_1 \phi''(1) + o(1))$, then plugging in
$r_1 = 2 / (p_c^3 \phi''(1))$ to obtain $\ee (\mu + o(1))$.
Similarly, the numerator is equal to $(r_1 \ee + o(\ee)) (\mu + o(1)) 
= (\mu r_1 + o(1)) \ee$.  Dividing yields $g'(p_c + \ee) = r_1 + o(\ee)$,
verifying the proposition when $k=1$.

Next, inserting our assumption of $2k+1$ moments into 
Proposition~\ref{pr:g-rec-rel} gives the order-$2k$ expansion
$$g(p_c + \ee) = r_1 \ee + \cdots + r_{2k} \ee^{2k} + o(\ee^{2k})$$
as $\ee \downarrow 0$.  We assume for induction that $g^{(j)} (p) 
\to j! \, r_j$ as $p \downarrow p_c$ for $1 \leq j < k$ and must
prove it for $j=k$.  We claim it is enough to show that 
$g^{(k+1)} (p_c + \ee) = O(\ee^{-2})$.  To see why, assume that
this holds but that $g^{(k)} (p)$ fails to converge to $k! r_k$ 
as $p \downarrow p_c$.  Setting $N = 2k$, the conclusion of 
Lemma~\ref{lem:g-upper-bound} would then yield a sequence of values
$\ee_n \downarrow 0$ with $g^{(k+1)} (p_c + \ee_n) / \ee_n^{-2}$ 
tending to infinity, which would be a contradiction.  

To show $g^{(k+1)} (p_c + \ee) = O(\ee^{-2})$ we begin by showing
that $g^{(k)} (p_c + \ee) = O(\ee^{-1})$; both arguments are similar
and the result for $g^{(k)}$ is needed for $g^{(k+1)}$.
Again we begin with the identity 
$g(p) = 1 - \phi(1 - p g(p))$, this time differentiating $k$ times.
The result is a sum of terms of the form 
$C p^a \phi^{(b)} (1 - p g(p)) \prod_{j=0}^k [g^{(j)} (p)]^{c_j}$.
To keep track of the proliferation of these under successive differentiation,
let the triple $\langle a, b, \cc\rangle$ to denote such a term, where $\cc = 
(c_0, \ldots , c_k)$ and we ignore the value of the multiplicative 
constant $C$.  For example, before differentiating at all, the right hand 
side is represented as $\langle 0 , 0 , \delta_0 \rangle$ where 
$\delta_i$ denotes the vector with $i$-component equal to~1 and the remaining
components equal to~0.  Differentiating replaces a term 
$\langle a, b, \cc \rangle$ with a sum of terms of four types, where 
$a, b, i$ and the entries of the vector $\cc := (c_0, \ldots c_k)$ 
are nonnegative integers and the last type can only occur when $c_i > 0$:
$$
\langle a-1, b, \cc \rangle \; , \; 
\langle a, b+1, \cc + \delta_0 \rangle \; , \; 
\langle a+1, b+1 , \cc + \delta_1 \rangle \; , \; \mbox{ and }
\langle a, b, \cc + \delta_{i+1} - \delta_i \rangle \; .
$$
After differentiating $k$ times,
the only term on the right-hand side for which $c_k \neq 0$
will be obtained from the term of the third type in the first 
differentiation, yielding $\langle 1, 1, \delta_1 \rangle$,
followed by the term of the fourth type in each successive differentiation,
yielding $\langle 1, 1, \delta_k \rangle$.  In other words,
the only summand on the right with a factor of $g^{(k)} (p)$ will
be the term $p \phi' (1 - p g(p)) g^{(k)} (p)$.  Subtracting this
from the left-hand side and dividing by $1 - p \phi' (1 - p g(p))$
results in an expression
$$g^{(k)}(p) = \frac{\sum \langle a, b, \cc \rangle}{1 - p \phi' (1-pg(p))}$$
where the summands in the numerator have $a \leq k, b \leq k$ and 
$c_k = 0$.  By the induction hypothesis, we know that the numerator 
is $O(1)$, and by~\eqref{eq:order}, the denominator is $\Theta (\ee)$,
proving that $g^{(k)}(p_c+\ee) = O(\ee^{-1})$. 

Identical reasoning with $k+1$ in place of $k$ shows that 
$$g^{(k+1)}(p) = \frac{\sum \langle a, b, \cc \rangle}{1 - p \phi' (1-pg(p))}$$
where the summands have $a \leq k+1, b \leq k+1, c_{k+1} = 0$ and
$c_k = 0$ or~1.  By~\eqref{eq:order} we know that $g^{(k)} (p_c + \ee) 
= O(\ee^{-1})$.  We conclude this time that the numerator is $O(\ee^{-1})$.  
The denominator is still $\Theta (\ee^{-1})$, allowing us to conclude
that $g^{(k+1)} (p_c + \ee) = O(\ee^{-2})$ and completing the induction.
$\Cox$

\section{Proof of part~$(ii)$ of Theorem~\protect{\ref{th:main}}: 
behavior at criticality}    \label{sec:critical}

This section is concerned with the expansion of $g(\TT , \cdot)$
at criticality.  Section~\ref{ss:expansion} defines the quantities that
yield the expansion.  Section~\ref{ss:martingales} constructs some
martingales and asymptotically identifies the expected number 
of $k$-subsets of $T_n$ that survive critical percolation as
a polynomial of degree $k-1$ whose leading term is a constant 
multiple of $W$ (a consequence of Theorem~\ref{th:Y-jk-mart}, below).  
Section~\ref{ss:diff} finishes computing the $\ell$-term Taylor expansion 
for $g(\TT,\cdot)$ at criticality.

\subsection{Explicit expansion} \label{ss:expansion}

Throughout the paper we use $\{ r_j \}$ to denote the coefficients
of the expansion of $g$ when they exist, given by the explicit 
formula~\eqref{eq:r}.  For $m \geq 1$, the $m$th power of $g$ has
a $k$-order expansion at $p_c^+$ whenever $g$ does.  Generalizing
the notation for $r_j$, we denote the coefficients of the expansion
of $g^m$ at $p_c^+$ by $\{ \rmj \}$ where
\begin{equation} \label{eq:rmj}
g(p_c + \ee)^m = \sum_{j=1}^\ell \rmj \ee^j + o(\ee^\ell) 
\end{equation}
for any $\ell$ for which such an expansion exists.  

We prove part~$(ii)$ of Theorem~\ref{th:main} by identifying the 
expansion.  To do so, we need a notation for certain expectations.
Fix a tree $T$.  For $n \geq 0$, $j\geq1$ and $v \in T$ , define
$$X_n^{(j)}(v) := \sum \limits_{\{v_1,\ldots,v_j\} \in \binom{T_n(v)}{j}}
   \P_T[v \conn_{p_c} v_1 ,v_2,\ldots, v_j]$$
where $v \conn_{p_c}v_1, v_2 ,\ldots ,v_j$ is the event that
$v$ is connected to each of $v_1,\ldots, v_j$ under critical percolation.
We omit the argument $v$ when it is the root; thus $X_n^{(j)} 
:= X_n^{(j)}(\rtt)$.  Note that
$$X_n^{(1)} = W_n,\qquad \text{and}\qquad X_n^{(2)}
   = \sum_{\{u,v\} \in \binom{T_n}{2}} p_c^{2n - |u \wedge v|}\,.$$
The former is the familiar martingale associated to a branching process,
while the latter is related to the energy of the uniform measure on $T_n$.

Extend this definition further: for integers $j$ and $k$, define
$$X_{n}^{(j,k)} := \sum_{\{v_i\} \in \binom{T_n}{j}}
   \binom{|T(v_1,\ldots,v_j)|}{k} p_c^{|T(v_1,\ldots,v_j)|}$$
where $T(v_1,\ldots,v_j)$ is the smallest rooted subtree of $T$ containing
each $v_i$ and $|T(v_1,\ldots,v_j)|$ refers to the number of edges
this subtree contains.  Note that $X_n^{(j,0)} = X_n^{(j)}$.

Part~$(ii)$ of Theorem~\ref{th:main} follows immediately from the 
following expansion, which is the main work of this section.

\begin{thm} \label{th:g-expansion}
Define
\begin{equation} \label{eq:M}
M_n^{(i)} := M_n^{(i)} (T) := \mu^i \sum_{j = 1}^{i} (-1)^{j+1} \sum_{d=j}^i 
   p_c^d r_{j,d} X_n^{(j,i-d)} \, . 
\end{equation}
Suppose that $\E \left [ Z^{(2 \ell + 1)(1 + \beta)} \right ] < \infty$
for some integer $\ell \geq 1$ and real $\beta > 0$.  $(i)$ The quantities
$\{ M_n^{(i)} : n \geq 1 \}$ are a $\{ \T_n \}$-martingale with mean $r_i$. 
$(ii)$ For $\GW$-almost every tree $T$ the limits 
$M^{(i)} := \lim_{n \to \infty} M_n^{(i)}$ exist.  $(iii)$ These 
limits are the coefficients in the expansion
\begin{equation} \label{eq:g-expansion}
g(T , p_c + \ee) = \sum_{i=1}^\ell M^{(i)} \ee^i + o(\ee^\ell) \, .
\end{equation}
\end{thm}

\begin{unremark}
The quantities $X_n^{(j,i)}$ do not themselves have limits as 
$n \to \infty$.  In fact for fixed $i$ and $j$ the sum over $d$
of $X_n^{(j,i-d)}$ is of order $n^{i-1}$.  Therefore it is important
to take the alternating outer sum before taking the limit.  
\end{unremark}

\subsection{Critical Survival of $k$-Sets} \label{ss:martingales}

To prove Theorem~\ref{th:g-expansion} we need to work with centered
variables.  Centering at the unconditional expectation is not good
enough because these mean zero differences are close to the nondegenerate
random variable $n^{i-1} W$ and therefore not summable.  Instead we 
subtract off a quantity that can be handled combinatorially, leaving
a convergent martingale.

Throughout the rest of the paper, the notation $\Delta$ in front of a 
random variable with a subscript (and possibly superscripts as well)
denotes the backward difference in the subscripted variable.  Thus,
for example, 
$$\Delta X_n^{(j,i)} := X_n^{(j,i)} - X_{n-1}^{(j,i)} \, .$$

Let $X_n^{(j,i)} = Y_n^{(j,i)} + A_n^{(j,i)}$ denote the Doob decomposition 
of the process $\{ X_n^{(j,i)} : n = 1, 2, 3, \ldots \}$ on the filtration
$\{ \T_n \}$.  To recall what this means, ignoring superscripts for 
a moment, the $Y$ and $A$ processes are uniquely determined by requiring 
the $Y$ process to be a martingale and the $A$ process to be predictable, 
meaning that $A_n \in \T_{n-1}$ and $A_0 = 0$.  The decomposition can be 
constructed inductively in $n$ by letting $A_0 = 0, Y_0 = \E X_0$ 
and defining 
\begin{eqnarray*}
\Delta A_n & := & \E \left ( \Delta X_n \| \T_{n-1} \right ) \, ; \\
\Delta Y_n & := & \Delta X_n - \Delta A_n \, .
\end{eqnarray*}
We begin by identifying the predictable part.

\begin{lem} \label{lem:predictable}
Let $\CC_i (j)$ denote the set of compositions of $j$ of length $i$
into strictly positive parts.  Let $m_r := \E 
\binom{Z}{r}$ and define constants $c_{j,i}$ by
$$c_{j,i} := p_c^j \sum_{\alpha \in \mathcal{C}_i(j)} 
   m_{\alpha_1} m_{\alpha_2}\cdots m_{\alpha_i} \, .$$ 
Then for each $k \geq 0$,
\begin{eqnarray} 
\Delta A_{n+1}^{(j,k)} & = & 
   -X_n^{(j,k)} 
   + \sum_{i=1}^j c_{j,i} \sum_{d=0}^k \binom{j}{k-d} X_n^{(i,d)} \nonumber \\
& = & \sum_{d=0}^{k-1} \binom{j}{k-d} X_n^{(j,d)} 
   + \sum_{i=1}^{j-1} \sum_{d=0}^k c_{j,i} \binom{j}{k-d} X_n^{(i,d)} 
   \label{eq:A} \, .
\end{eqnarray}
\end{lem}

\begin{proof}
For distinct vertices $v_1,\ldots,v_j$ in $\TT_{n+1}$, 
their set of parents $u_1,\ldots,u_\ell$ form a subset of $\TT_n$ 
with at most $j$ elements.  In order to sum over all $j$-sets of 
$\TT_{n+1}$, one first sums over all sets of parents.  For a fixed 
parent set $u_1,\ldots,u_\ell$ in $\TT_{n-1}$, the total number of 
$j$-sets with parent set $\{u_1,\ldots,u_\ell\}$ is 
$$\sum_{\alpha \in \mathcal{C}_\ell(j)} 
   \binom{Z_1(u_1)}{\alpha_1}\cdots 
   \binom{Z_1(u_\ell)}{\alpha_\ell} \, .$$  
Furthermore, we have 
$$\binom{|\TT(v_1,\ldots,v_j)|}{k} 
   = \binom{|\TT(u_1,\ldots,u_\ell)| + j}{k} 
   = \sum_{d = 0}^k \binom{j}{k-d}\binom{|\TT(u_1,\ldots,u_\ell)|}{d} 
\, .$$
This gives the expansion 
\begin{align*}
X_{n+1}^{(j,k)} 
&= \sum_{\{v_i\} \in \binom{\TT_{n+1}}{j}} 
   \binom{|\TT(v_1,\ldots,v_j)|}{k} p_c^{|\TT(v_1,\ldots,v_j)|} \\
&= \sum_{\ell = 1}^j \sum_{\{u_i\} \in \binom{\TT_n}{\ell}}
   \sum_{d = 0}^k \binom{j}{k-d} \binom{\TT(u_1,\ldots,u_\ell)}{d} 
   p_c^{|\TT(u_1,\ldots,u_\ell)|} 
   \sum_{\alpha \in \mathcal{C}_\ell(j)}  
   p_c^j \binom{Z_1(u_1)}{\alpha_1}\cdots \binom{Z_1(u_\ell)}{\alpha_\ell}
\, . 
\end{align*}
Taking conditional expectations with respect to $\T_n$ completes the proof
of the first identity, with the second following from rearrangement of terms.
$\Cox$
\end{proof}

The following corollary is immediate from Lemma~\ref{lem:predictable}
and the fact that $X_0^{(j,k)} = Y_0^{(j,k)}$.

\begin{cor} \label{cor:moment-martingale}
For each $j$ so that $\E[Z^j] < \infty$ and each $k$, the terms of the
$Y$ martingale are given by
\begin{eqnarray*}
Y_n^{(j,k)} & = & Y_0^{(j,k)} + \sum_{m=1}^n \Delta Y_m^{(j,k)} \\
& = & X_n^{(j,k)} - \sum_{m = 0}^{n-1}
   \left[\sum_{d = 0}^{k-1} \binom{j}{k-d}X_m^{(j,d)} 
   + \sum_{i = 1}^{j-1} c_{j,i} \sum_{d = 0}^k 
   \binom{j}{k-d} X_m^{(i,d)} \right] \, .
\end{eqnarray*}
$\Cox$
\end{cor}

We want to show that these martingales converge both almost surely and in some appropriate $L^p$ space; this will require us to take $L^{1+\beta}$ norms for some $\beta \in (0,1]$.  The following randomized version of the 
Marcinkiewicz-Zygmund inequality will be useful.

\begin{lem} \label{lem:moment-sum}
Let $\{\xi_k \}_{k=1}^\infty$ be i.i.d. with $\E[\xi_1] = 0$ 
and $\E[|\xi_1|^{1+\beta}] < \infty$ for some $\beta \in (0,1]$, 
and let $N$ be a random variable in $\mathbb{N}$ independent from 
all $\{\xi_k\}$ and with $\E[N] < \infty$.  If we set 
$S_n = \sum_{k=1}^n \xi_k$, then there exists a constant $c > 0$ 
depending only on $\beta$ so that 
$$\E[|S_N|^{1+\beta}] \leq c \E[|\xi_1|^{1+\beta}] \E[N] \,.$$

In particular, if $\xi(v)$ are associated to vertices $v \in \TT_s$, 
and are mutually independent from $\TT_s$, then 
$$\left\Vert p_c^s \sum_{v \in \TT_s} \xi(v) \right\Vert_{L^{1 + \beta}} 
   \leq c' p_c^{s\beta/(1 + \beta)} \Vert\xi(v) \Vert_{L^{1 + \beta}} \, .$$
\end{lem}

\begin{proof}
Suppose first that $N$ is identically equal to a constant $n$.  The 
Marcinkiewicz-Zygmund inequality (e.g. \cite[Theorem 10.3.2]{chow-teicher}) 
implies that there exists a constant $c > 0$ depending only on $\beta$ 
such that 
$$\E[|S_n|^{1+\beta}] 
   \leq c \E\left[\left(\sum_{k=1}^n |\xi_k|^2 \right)^{(1+\beta)/2}\right]
\, . $$
	
Because $1 + \beta \leq 2$ and the $\ell^p$ norms descend, we have 
$\Vert (\xi_k)_{k=1}^n \Vert_{\ell^2} \leq \Vert (\xi_k)_{k=1}^n 
\Vert_{\ell^{1+\beta}}$ deterministically; this completes the proof 
when $N$ is constant.  Writing $\E[|S_N|^{1+\beta}] = 
\E\left[\E[|S_N|^{1+\beta} \| N] \right]$ and applying 
the bound from the constant case completes the proof.	
$\Cox$
\end{proof}

We now show that the martingales $\{ Y_n^{(j,k)} : n \geq 0 \}$ converge.

\begin{thm} \label{th:Y-jk-mart}
Suppose $\E[Z^{j(1 + \beta)}] < \infty$ for some $\beta > 0$.  Then 
\begin{enumerate}[$(a)$]
\item $\Vert \Delta Y_{n+1}^{(j,k)} \Vert_{L^{1 + \beta}} 
\leq C e^{-c n}$ where $C$ and $c$ are positive constants 
depending on $j,k,\beta$ and the offspring distribution.
\item $Y_n^{(j,k)}$ converges almost surely and in $L^{1+\beta}$
to a limit, which we denote $Y^{(j,k)}$.
\item There exists a positive constant $c_{j,k}'$ depending only on $j,k$ 
and the offspring distribution so that 

$X_n^{(j,k)} n^{-(j+k-1)} \to c_{j,k}' W$ 
almost surely and in $L^{1 + \beta}$.
\end{enumerate}
\end{thm}

\begin{proof} ~~\\
\ul{Step 1: $(a) \implies (b)$}.  For any fixed $j$ and $k$: 
the triangle inequality and $(a)$ show that $\sup_n \Vert Y_n^{(j,k)} 
\Vert_{L^{1+\beta}} < \infty$, from which $(b)$ follows from the 
$L^p$ martingale convergence theorem.  Next, we prove an identity 
representing $X_n^{(j,k)}$ as a multiple sum over values of $X^{(j' , k')}$
with $(j' , k') < (j,k)$ lexicographically.  

\ul{Step 2: Some computation}.
For a set of vertices $\{v_1,\ldots,v_j\}$, let $v = v_1 \wedge v_2 
\wedge \cdots \wedge v_j$ denote their most recent common ancestor.  
In order for $\rtt \conn_{p_c} v_1,\ldots, v_j$ to hold, we must first 
have $\rtt \conn_{p_c} v$.  For the case of $j \geq 2$, looking at 
the smallest tree containing $v$ and $\{v_i\}$, we must have that 
this tree branches into some number of children $a \in [2,j]$ 
immediately after $v$.  We may thus sum over all possible $v$, 
first by height, setting $s = |v|$, then choosing how many children 
of $v$ will be the ancestors of the $v_1,\ldots,v_j$.  We then choose 
those children $\{u_r\}$, and choose how to distribute the 
$\{v_\ell\}$ among them.  
In order for critical percolation to reach each $v_1,\ldots, v_j$, 
it must first reach $v$, then survive to each child of $v$ that 
is an ancestor of some $\{v_\ell \}$ and then survive to the 
$\{v_\ell \}$ from there.  Finally, in order to choose the 
$k$-element subset corresponding to $\binom{|\TT_(v_1,\ldots,v_j)|}{k}$, 
we may choose $\alpha_0$ elements from the tree $T(u_1,\ldots,u_a)$ 
and $\alpha_\ell$ elements from each subtree of $u_\ell$.  Putting 
this all together, we have the decomposition
\begin{align}
X_n^{(j,k)} &= 
   \sum_{s = 0}^{n-1}p_c^s
   \sum_{v \in \TT_s}
   \sum_{a = 2}^j 
   \sum_{u \in \binom{\TT_1(v)}{a}} p_c^a 
   \sum_{\beta \in \mathcal{C}_a(j)}
   \sum_{\alpha_0 = 0}^{k} 
   \sum_{\alpha \in \widetilde{\mathcal{C}}_a(k-\alpha_0)} 
      \binom{s+a}{\alpha_0} X_{n-s-1}^{(\beta_1,\alpha_1)}(u_1) 
      \cdots X_{n-s-1}^{(\beta_a,\alpha_a)}(u_a) \label{eq:X_n-decomp} \\
&= \sum_{s = 0}^{n-1}p_c^s \sum_{v \in \TT_s} \Theta_{n - s - 1}^{(j,k)}(v) 
   \nonumber
\end{align}
where $\Theta_{n-s-1}^{(j,k)}(v)$ is defined as the inner quintuple sum 
in the previous line and $\widetilde{\mathcal{C}}_a(k)$ denotes the 
set of weak $a$-compositions of $k$; observe that the notation 
$\Theta_{n-s-1}^{(j,k)}(v)$ hides the dependence on $s = |v|$.  

The difference $\Delta Y_n^{(j,k)}$ can now be computed as follows:
\begin{eqnarray} \label{eq:Y-n-big-diff}
\Delta Y_n^{(j,k)} & = & X_n^{(j,k)} - \sum_{i = 1}^j 
   \sum_{d = 0}^k  \binom{j}{k-d}c_{j,i} X_{n-1}^{(i,d)} \\
& = & \sum_{s = 0}^{n-1} p_c^s \sum_{v \in \TT_s} \Theta_{n-s-1}^{(j,k)}(v) 
   - \sum_{s = 0}^{n-2} p_c^s \sum_{v \in \TT_s}\sum_{i=1}^j 
   \sum_{d=0}^k \binom{j}{k-d} c_{j,i} \Theta_{n-s-2}^{(i,d)}(v) \nonumber \\
& = &\sum_{s=0}^{n-2} p_c^s \sum_{v \in \TT_s}
   \left(\Theta_{n-s-1}^{(j,k)}(v) - \sum_{i=2}^j 
   \sum_{d=0}^k \binom{j}{k-d} c_{j,i} \Theta_{n-s-2}^{(i,d)}(v)
   \right)\nonumber \\
&&\qquad + \left[p_c^{n-1}\sum_{v \in \TT_{n-1}} 
   \Theta_{0}^{(j,k)}(v) -  c_{j,1} \sum_{d = 0}^k 
   \binom{j}{k-d}X_{n-1}^{(1,d)} \right]\, \nonumber \\
& = &  \sum_{s=0}^{n-2} p_c^s \sum_{v \in \TT_s} U_n^{(j,k)} (v) + V_n^{(j,k)}
   \nonumber
\end{eqnarray}
where  
\begin{eqnarray}
U_n^{(j,k)} & = & \left(\Theta_{n-s-1}^{(j,k)}(v) - \sum_{i=2}^j \sum_{d=0}^k 
\binom{j}{k-d} c_{j,i} \Theta_{n-s-2}^{(i,d)}(v)\right)
\label{eq:U} \, ;\\
V_n^{(j,k)} & = & \left(p_c^{n-1} \sum_{v \in \TT_{n-1}} p_c^j 
   \binom{n+j-1}{k} \binom{Z_1(v)}{j} \right) 
   -  c_{j,1} \sum_{d = 0}^k \binom{j}{k-d}X_{n-1}^{(1,d)} 
   \, . \label{eq:V} 
\end{eqnarray}

\ul{Step 3: Proving $(a)$ and $(c)$ for $j=1$ and $k$ arbitrary}.  
Specializing~\eqref{eq:Y-n-big-diff} to $j=1$ yields  
\begin{eqnarray*}
\Delta Y_n^{(1,k)}
   & = & \binom{n}{k}W_n - \binom{n-1}{k}W_{n-1} - \binom{n-1}{k-1}W_{n-1} \\
   & = & \binom{n-1}{k}(W_n - W_{n-1}) + \binom{n-1}{k-1}(W_n - W_{n-1}) \, .
\end{eqnarray*}
The quantity $W_n - W_{n-1}$ is the sum of independent contributions
below each vertex in $T_{n-1}$; Lemma~\ref{lem:moment-sum} shows this
to be exponentially small in $L^{1+\beta}$ and proving $(a)$, hence $(b)$. 
Additionally, $Y_n^{(1,k)}n^{-k} \to W/k!$, thereby also showing $(c)$ for 
$j = 1$ and all $k$. 

\ul{Step 4: $V$ is always small}.  Using the identity
$\binom{n+j -1}{k} = \sum_{d = 0}^{k} \binom{n-1}{d} \binom{j}{k-d}$ 
and recalling that $X_{n-1}^{(1,d)} = \binom{n-1}{d}W_{n-1}$ shows that
$$V_n^{(j,k)} = \sum_{d = 0}^k\binom{j}{k-d}\binom{n-1}{d} 
   p_c^{n-1}\sum_{v \in \TT_{n-1}}  p_c^j 
   \left[\binom{Z_1(v)}{j} - \E\binom{Z}{j}\right]\,. $$
Applying Lemma \ref{lem:moment-sum} shows that the innermost sum, 
when multiplied by $p_c^{n-1}$, has $L^{1+\beta}$ norm that is 
exponentially small in $n$.  With $k$ fixed and $d \leq k$, the
product with $\binom{n-1}{d}$ still yields an exponentially small 
variable, thus 
\begin{equation} \label{eq:V small}
||V_n^{(j,k)}||_{1+\beta} \leq c_{j,k,\beta} e^{-\delta n}
\end{equation}
for some $\delta = \delta (j,k,\beta) > 0$.  

The remainder of the proof is an induction in two stages (Steps~5 and~6).  
In the first stage we fix $j > 1$, assume $(a)$--$(c)$ for all $(j' , k')$ 
with $j' < j$, and prove $(a)$ for $(j,k)$ with $k$ arbitrary.  In the
second stage, we prove $(c)$ for $(j,k)$ by induction on $k$,
establishing $(c)$ for $(j,1)$ and then for arbitrary $k$ by 
induction, assuming $(a)$ for $(j,k')$ where $k'$ is arbitrary 
and $(c)$ for $(j,k')$ where $k' < k$.

\ul{Step 5: Prove~$(a)$ by induction on $j$}.  Fix $j \geq 2$ and assume 
for induction that $(a)$ and $(c)$ hold for all $(j' , k)$ with $j' < j$.
The plan is this:
The quantity $p_c^s \sum_{v \in \TT_s} U_n^{(j,k)} (v)$ is $W_n$ times
the average of $U_n^{(j,k)} (v)$ over vertices $v \in \TT_s$.  Averaging
many mean zero terms will produce something exponentially small in $s$.
We will also show this quantity to be also exponentially small in $n-s$,
whereby it follows that the outer sum over $s$ is exponentially small, 
completing the proof. 

Let us first see that $U_n^{(j,k)} (v)$ has mean zero.  Expanding back
the $\Theta$ terms gives
\begin{eqnarray}
U_n^{(j,k)} (v) & = & 
\sum_{a = 2}^j \sum_{u \in \binom{\TT_1(v)}{a}}p_c^a 
   \sum_{\alpha_0 = 0}^{k} \binom{s + a}{\alpha_0} 
   \Bigg(\sum_{\beta \in \mathcal{C}_a(j)} 
   \sum_{\alpha \in \widetilde{\mathcal{C}}_a 
   (k - \alpha_0)} X_{n-s-1}^{(\beta_1,\alpha_1)}(u_1) \cdots 
   X_{n-s-1}^{(\beta_a,\alpha_a)}(u_a) \nonumber \\
&& - \sum_{i =2}^{j}\sum_{d = 0}^k c_{j,i} \binom{j}{k-d} 
   \sum_{\beta' \in \mathcal{C}_a(i)} 
   \sum_{\alpha' \in \widetilde{\mathcal{C}}_a(d - \alpha_0)} 
   X_{n-s-2}^{(\beta_1',\alpha_1')}(u_1) \cdots 
   X_{n-s-2}^{(\beta_a',\alpha_a')}(u_a)\Bigg)\,. \label{eq:summand}
\end{eqnarray}
Expanding the first product of $X$ terms gives
\begin{equation} \label{eq:X terms}
X_{n-s-1}^{(\beta_1,\alpha_1)}(u_1) 
   \cdots X_{n-s-1}^{(\beta_a,\alpha_a)}(u_a) =
   \prod_{\ell = 1}^a \left(\Delta Y_{n-s-1}^{(\beta_\ell,\alpha_\ell)}(u_\ell)
   + \sum_{\beta_\ell' = 1}^{\beta_\ell} c_{\beta_\ell,\beta_\ell'} 
   \sum_{\alpha_\ell'=0}^{\alpha_\ell} 
   \binom{\beta_\ell}{\alpha_\ell - \alpha_\ell'} 
   X_{n-s-2}^{(\beta_\ell',\alpha_\ell')}(u_\ell)\right) \, .
\end{equation}
The vertices $u_\ell$ are all distinct children of $v$.  Therefore,
their subtrees are jointly independent, hence the pairs $(\Delta Y (u_\ell) ,
X (u_\ell))$ are jointly independent.  The product~\eqref{eq:X terms} 
expands to the sum of $a$-fold products of terms, each term
in each product being either a $\Delta Y$ or a weighted sum of $X$'s,
the $a$ terms being jointly independent by the previous observation.  
Therefore, to see that the whole thing is mean
zero, we need to check that the product of the $a$ different sums of 
$X$ terms in~\eqref{eq:X terms}, summed over $\alpha$ and $\beta$ to form 
the first half of the summand in~\eqref{eq:summand}, minus the subsequent 
sum over $i, d, \beta'$ and $\alpha'$, has mean zero.  In fact we will show
that it vanishes entirely.  For given compositions $\beta := 
(\beta_1 , \ldots , \beta_a)$ and $\alpha := (\alpha_1, \ldots , \alpha_a)$, 
the product of the double sum of $X$ terms inside the round brackets 
in~\eqref{eq:X terms} may be simplified:
$$\prod_{\ell = 1}^a \left( 
   \sum_{\beta_\ell' = 1}^{\beta_\ell} c_{\beta_\ell,\beta_\ell'} 
   \sum_{\alpha_\ell'=0}^{\alpha_\ell}
   \binom{\beta_\ell}{\alpha_\ell - \alpha_\ell'} 
   X_{n-s-2}^{(\beta_\ell',\alpha_\ell')}(u_\ell)\right) \\
= \sum_{1 \preceq \beta' \preceq \beta} \sum_{0 \preceq \alpha' \preceq \alpha} 
   \prod_{\ell=1}^a c_{\beta_\ell,\beta_\ell'}
   \binom{\beta_\ell}{\alpha_\ell - \alpha_\ell'}
   X_{n-s-2}^{(\beta_\ell',\alpha_\ell')}(u_\ell) \, . 
$$
Applying the identity
\begin{equation} \label{eq:compositions}
\sum_{\substack{\beta \in \mathcal{C}_a(j) \\ \beta \succeq \beta'}}
   \left(\prod_{\ell} c_{\beta_\ell,\beta'_\ell}\right) = c_{j,i} \, ,
\end{equation}
which follows by regrouping pieces of each composition in 
$\mathcal{C}_i(j)$ into smaller compositions each with $\beta_\ell'$ parts,
then summing over $\alpha$ and $\beta$ as in~\eqref{eq:summand}
and simplifying, using~\eqref{eq:compositions} in the last line,
gives
\begin{align*} 
   \sum_{\beta \in \mathcal{C}_a(j)} \sum_{1 \preceq \beta' \preceq \beta} 
& \sum_{\alpha \in \widetilde{\mathcal{C}}_a(k-\alpha_0)} 
   \sum_{0 \preceq \alpha' \preceq \alpha}
   \prod_{\ell=1}^a c_{\beta_\ell,\beta_\ell'}
   \binom{\beta_\ell}{\alpha_\ell - \alpha_\ell'}
   X_{n-s-2}^{(\beta_\ell',\alpha_\ell')}(u_\ell) \\
& = \sum_{\beta \in \mathcal{C}_a(j)} \sum_{1 \preceq \beta' \preceq \beta} 
   \left(\prod_{\ell} c_{\beta_\ell,\beta'_\ell}\right)
   \sum_{d = 0}^k \sum_{\alpha' \in \widetilde{\mathcal{C}}_a(d - \alpha_0) } 
   \left(\prod_{\ell} X_{n-s-2}^{(\beta_\ell',\alpha_\ell')}(u_\ell)\right) 
   \sum_{ \substack{\alpha \in \widetilde{\mathcal{C}}_a(k-\alpha_0) \\ 
         \alpha \geq \alpha'}}
   \prod_{\ell=1}^a \binom{\beta_\ell}{\alpha_\ell - \alpha_\ell'} \\
& = \sum_{\beta \in \mathcal{C}_a(j)} \sum_{1 \preceq \beta' \preceq \beta} 
   \left(\prod_{\ell} c_{\beta_\ell,\beta'_\ell}\right)
   \sum_{d = 0}^k \sum_{\alpha' \in \widetilde{\mathcal{C}}_a(d - \alpha_0) } 
   \left(\prod_{\ell} X_{n-s-2}^{(\beta_\ell',\alpha_\ell')}(u_\ell)\right) 
   \binom{j}{k - d} \\
& = \sum_{i = 2}^j \sum_{\beta' \in \mathcal{C}_a(i)} 
   \sum_{\substack{\beta \in \mathcal{C}_a(j) \\ \beta \geq \beta'}}
   \left(\prod_{\ell} c_{\beta_\ell,\beta'_\ell}\right) 
   \sum_{d = 0}^k \sum_{\alpha' \in \widetilde{\mathcal{C}}_a(d - \alpha_0) } 
   \left(\prod_{\ell} X_{n-s-2}^{(\beta_\ell',\alpha_\ell')}(u_\ell)\right) 
   \binom{j}{k - d} \\
& = \sum_{i =2}^{j}\sum_{d = 0}^k c_{j,i} \binom{j}{k-d} 
   \sum_{\beta' \in \mathcal{C}_a(i)} 
   \sum_{\alpha' \in \widetilde{\mathcal{C}}_a(d - \alpha_0)} 
   X_{n-s-2}^{(\beta_1',\alpha_1')}(u_1) \cdots  
      X_{n-s-2}^{(\beta_a',\alpha_a')}(u_a) \, .
\end{align*}
This exactly cancels with the quadruple sum on the second line 
of~\eqref{eq:summand}, transforming~\eqref{eq:summand} into
$$U_n^{(j,k)} (v) =
\sum_{a = 2}^j \sum_{u \in \binom{\TT_1(v)}{a}}p_c^a
   \sum_{\alpha_0 = 0}^{k} \binom{s + a}{\alpha_0}
   \sum_{\beta \in \mathcal{C}_a(j)} 
   \sum_{\alpha \in \widetilde{\mathcal{C}}_a (k - \alpha_0)} 
   \prod_{\ell = 1}^a (*)_\ell \, , $$
where $(*)_\ell = \Delta Y_{n-s-1}^{(\beta_\ell , \alpha_\ell)} (u_\ell)$ 
for at least one value of $\ell$ in $[1,a]$, and, when not equal to that, 
is equal to the last double sum inside the brackets in~\eqref{eq:X terms}.

By the induction hypothesis, the $\Delta Y$ terms have $(1+\beta)$ 
norm bounded above by something exponentially small:
\begin{equation} \label{eq:kappa}
\lVert\Delta Y_{n-s-1}^{(\beta_\ell , \alpha_\ell)} (u_\ell) \rVert
   = O\left(\exp \left [ - \kappa_{\beta_\ell , \alpha_\ell} (n-s-1) \right ]\right) 
   \, .
\end{equation}
We note that $a$ and each $\beta_\ell$ and $\alpha_\ell$ are all bounded
above by $j$ and that in each product $X_{n-s-1}^{(\beta_1,\alpha_1)}(u_1) 
\cdots X_{n-s-1}^{(\beta_a,\alpha_a)}(u_a)$, the terms are independent.
The inductive hypothesis implies each factor $X_n^{(j,k)}$ has $L^{1+\beta}$ norm that is $O(n^{\lambda(j,k)})$.

Returning to \eqref{eq:Y-n-big-diff}, we may apply 
Lemma~\ref{lem:moment-sum} to see that for each $s$, the 
quantity $p_c^s \sum_{v \in \TT_s} U_n^{(j,k)} (v)$ is an
average of $|\TT_s|$ terms all having mean zero and $L^{1+\beta}$
bound exponentially small in $n-s$, and that averaging introduces
another exponentially small factor, $\exp(-\nu s)$.  Because
the constants $\kappa, \lambda$ and $\mu$ vary over a set of 
bounded cardinality, the product of these three upper bounds, 
$O\left(\exp (-\kappa (n-s)) \cdot \exp (- \nu s) \cdot n^{\lambda (j,k)} \right)$ 
decreases exponentially $n$.

\ul{Step 6: Prove $(c)$ by induction on $(j,k)$}.  The final stage 
of the induction is to assume $(a)$--$(c)$ for $(j' , k')$
lexicographically smaller than $(j,k)$ and prove $(c)$ for $(j,k)$.
We use the following easy fact.
\begin{lem} \label{lem:partial}
If $a_n \to \infty$ and $a_n \sim b_n$ then the partials sums are
also asymptotically equivalent: $\sum_{k=1}^n a_k \sim \sum_{k=1}^n b_k$.
$\Cox$
\end{lem}

We begin the inductive proof of with the case $k=0$.  Rearranging the
conclusion of Corollary~\ref{cor:moment-martingale}, we see that
$$X_n^{(j,0)} = Y_n^{(j,0)} + \sum_{m=0}^{n-1} \sum_{i=1}^{j-1} X_m^{(i,0)} \, .$$
Using Lemma~\ref{lem:partial} the induction hypothesis, and the fact
that $Y_n^{(j,0)} = O(1)$ simplifies this to
\begin{align*}
X_n^{(j,0)} &\sim \sum_{m = 0}^{n-1} \left[ \sum_{i=1}^{j-1} m^{i-1} c_i' W \right]  \\
&\sim \sum_{m=0}^{n-1} c_{j-1}' m^{j-2} W \\
&\sim c_j' n^{j-1}W 
\end{align*}
where $c_j' = \lim_{n \to \infty} \frac{c_{j-1}'}{n} \sum_{m=0}^{n-1} (m/n)^{j-2}
= c_{j-1}' / (j-1)$.

The base case $k=0$ being complete, we induct on $k$.  The same reasoning,
observing that the first inner sum is dominated by the $d=k-1$ term
and the second by the $i = j-1$ and $d=k$ term, gives
\begin{align*}
X_n^{(j,k)} &=  Y_n^{(j,k)} + \sum_{m=0}^{n-1}\left[\sum_{d = 0}^{k-1} \binom{j}{k-d} X_m^{(j,d)} + \sum_{i = 1}^{j-1}c_{j,i}\sum_{d = 0}^k \binom{j}{k-d} X_m^{(i,d)} \right] \\
&\sim \sum_{m = 0}^{n-1}\left(j X_m^{(j,k-1)} + c_{j,j-1} X_m^{(j-1,k)}\right) \\
&\sim \sum_{m = 0}^{n-1} \left[j m^{j+k-2} W c_{j,k-1}' + c_{j,j-1}c_{j-1,k}' W m^{j+k-2}\right] \\
&\sim W \left(\frac{j c_{j,k-1}' + c_{j,j-1}c_{j-1,k}' }{j+k-1}  \right) n^{j+k-1}\,.
\end{align*}
Setting $\disp c_{j,k}' := \frac{j c_{j,k-1}' + c_{j,j-1}c_{j-1,k}' }{j+k-1}$ completes the almost-sure part of $(c)$ by induction.  The $L^{1+\beta}$ portion is similar, but we need one more easy fact. \begin{lem}
	If $a_n \to \infty$ and $b_n \to 0$ then $\sum_{k=1}^n a_n b_n = o\left( \sum_{k=1}^n a_k \right)$. $\Cox$
\end{lem} 
This allows us to calculate
\begin{align*}
\Big\lVert X_n^{(j,k)}&n^{-(j+k-1)} - W c_{j,k}' \Big\rVert_{L^{1+\beta}} \\
=&\left\lVert Y_n^{(j,k)}n^{-(j+k-1)} + n^{-(j+k-1)}\sum_{m=0}^{n-1}\left[\sum_{d = 0}^{k-1} \binom{j}{k-d} X_m^{(j,d)} + \sum_{i = 1}^{j-1}c_{j,i}\sum_{d = 0}^k \binom{j}{k-d} X_m^{(i,d)} \right]  - W c_{j,k}'\right\rVert_{L^{1+\beta}} \\
\leq&\ o(1) + n^{-(j+k-1)}\left\lVert \sum_{m = 0}^{n-1} \left(j X_m^{(j,k-1)} + c_{j,j-1} X_m^{(j-1,k)}\right) - n^{j+k-1}W c_{j,k}' \right\rVert_{L^{1+\beta}} \\
\leq&\ o(1) + n^{-(j + k -1)} \sum_{m=0}^{n-1} m^{j+k-2}\left( j \left\lVert \frac{X_m^{(j,k-1)}}{m^{j+k-2}} - W c_{j,k-1}' \right\rVert_{L^{1 + \beta}} + c_{j,j-1} \left\lVert \frac{X_m^{(j-1,k)}}{ m^{j+k-2}} - W c_{j-1,k}'\right\rVert_{L^{1+\beta}}\right) \\
=&\ o(1)\,.
\end{align*}
This completes the induction, and the proof of Theorem \ref{th:Y-jk-mart}.
$\Cox$
\end{proof}

\subsection{Expansion at Criticality} \label{ss:diff}

An easy inequality similar to classical Harris inequality~\cite{harris60}
is as follows.
\begin{lem}\label{lem:fkg}
For finite sets of edges $E_1,E_2,E_3$, define $A_j$ to be the event 
that all edges in $E_j$ are open.  Then 
$$\P[A_1 \cap A_2] \cdot \P[A_1 \cap A_3] 
   \leq \P[A_1] \cdot \P[A_1 \cap A_2 \cap A_3] \, .$$
\end{lem}

\begin{proof}
Writing each term explicitly, this is equivalent to the inequality 
$$p^{|E_1 \cup E_2| + |E_1 \cup E_3|} 
   \leq p^{|E_1| + |E_1 \cup E_2 \cup E_3|} \, .$$
Because $p \leq 1$, this is equivalent to 
$$|E_1 \cup E_2| + |E_1 \cup E_3| \geq |E_1| + |E_1 \cup E_2 \cup E_3| \, ,$$
which is easily proved for all triples $E_1, E_2, E_3$ by inclusion-exclusion.
$\Cox$
\end{proof}

Before finding the expansion at criticality, we show that focusing only
on the first $n$ levels of the tree and averaging over the remaining 
levels causes only a subpolynomial error in an appropriate sense.

\begin{pr}\label{pr:bonf-replace}
Suppose $\E[Z^{(2k -1)(1 + \beta)}] < \infty$, and set $p = p_c + \ee$.
Fix $\delta > 0$ and let $n = n(\ee) = \lceil\ee^{-\delta}\rceil$.  Then 
for $\delta$ sufficiently small and each $\ell > 0$, 
\begin{equation}\label{eq:average-error}
\sum_{\{u_i\} \in \binom{\TT_n}{k}}
   \P_\TT[\rtt \conn_{p} u_1,\ldots,u_k]
   \left(g(\TT(u_1),p)\cdots g(\TT(u_k),p) - g(p)^k \right) 
   = o(\ee^\ell)
\end{equation}
$\GW$-almost surely as $\ee \to 0^+$.
\end{pr}

\begin{proof}  For sufficiently small $\delta > 0$, we note that 
$(p_c + \ee)^{m} \leq 2 p_c^m$ for each $m \in [n,kn]$ and for $\ee$ 
sufficiently small.  This will be of use throughout, and is responsible 
for the appearance of factors of $2$ in the upper bounds.
	
Next, bound the variance of 
$$\sum_{\{u_i\} \in \binom{\TT_n}{k}}\P_\TT[\rtt \conn_{p} u_1,\ldots,u_k]
   \left[g(\TT(u_1),q)\cdots g(\TT(u_k),q) - g(q)^k \right]$$ 
for a fixed vertex, $q$.  This expression has mean zero conditioned on $\T_n$.
Its variance is equal to the expected value of its conditional variance 
given $\T_n$.  We therefore square and take the expectation, where
the second sum in the second and third lines are over pairs of disjoint  
$k$-tuples of points.
\begin{align*}
\E &\left[\left(\sum_{\{u_i\} \in \binom{\TT_n}{k}}
   \P_\TT[\rtt \conn_{p} u_1,\ldots,u_k]
   \left(g(\TT(u_1),q)\cdots g(\TT(u_k),q) - g(q)^k \right) \right)^2 
   \, \Bigg| \, \T_n \right] \\
& = \frac{1}{(k!)^2}\sum_{r = 1}^k r! 
   \sum_{\{u_i\}_{i=1}^k,\{v_i\}_{i=r+1}^k~\mathrm{dist.}} 
   \binom{k}{r}^2 \P_\TT[\rtt \conn_p u_1,\ldots,u_k]
   \P_\TT[\rtt \conn_p u_1,\ldots,u_r,v_{r+1},\ldots,v_k] C_r \\
& \leq \frac{1}{(k!)^2}\sum_{r = 1}^k r! 
   \sum_{\{u_i\}_{i=1}^k,\{v_i\}_{i=r+1}^k~\mathrm{dist.}} 
   \binom{k}{r}^2 \P_\TT[\rtt \conn_p u_1,\ldots,u_r]
   \P_\TT[\rtt \conn_p u_1,\ldots,u_k,v_{r+1},\ldots,v_k] C_r \\
& \leq 4 p_c^n \sum_{r = 1}^k \binom{2k - r}{k}\binom{k}{r}  
   C_r X_n^{(2k -r)} \, .
\end{align*} 
Here we have used the bounds $\P_\TT[\rtt \conn_p u_1,\ldots,u_r] 
\leq 2 p_c^n$ and $\P_\TT[\rtt \conn_p u_1,\ldots, v_k] \leq 
2 \P_\TT[\rtt \conn_{p_c} u_1,\ldots,v_k]$ and we have defined 
$$C_r := \E\left[\left(g(\TT(u_1),q)\cdots g(\TT(u_k),q) 
   - g(q)^k\right)\left(g(\TT(u_1),q)\cdots g(\TT(u_r),q)  
   g(\TT(v_{r+1}),q)\cdots g(\TT(v_{k}),q) - g(q)^k\right) \right] \, . $$
Taking the expected value and using Theorem \ref{th:Y-jk-mart} 
along with Jensen's Inequality and induction gives that the 
variance is bounded above by $C p_c^n n^{2k - 2}$ for some constant $C$.  
This is exponentially small in $n$, so there exist constants $c_k,
C_k > 0$ so that the variance is bounded above by $C_k e^{-c_k n}$.  
	
Define $a = a(m,r) = \frac{1}{m} + \frac{r}{m^{\ell+2}}$ and 
$b = b(m,r)= \frac{1}{m} + \frac{r +1}{m^{\ell+2}}$.  For each 
$\ee \in (0,1)$ there exists a unique pair $(m,r)$ such that 
$\ee \in [1/m, 1/(m-1))$ and $\ee \in [a,b)$. Assume for now 
that $\lceil a^{-\delta} \rceil = \lceil b^{-\delta} \rceil $; 
the case in which the two differ is handled at the end of the proof. 
For all $\ee \in [a,b)$ and $p = p_c + \ee$, we have 
\begin{align*}
\sum_{\{u_i\} \in \binom{\TT_n}{k}} \P_\TT
& \left[\rtt \conn_p u_1,\ldots,u_k\right] g(\TT(u_1),p)\cdots g(\TT(u_k),p) \\
& \leq \sum_{\{u_i\} \in \binom{\TT_n}{k}} \P_\TT
   \left[\rtt \conn_{p_c + b} u_1,\ldots,u_k\right] 
   g(\TT(u_1),p_c + b)\cdots g(\TT(u_k),p_c + b) \, .
\end{align*}
By Chebyshev's inequality, the conditional probability that the 
right-hand side is $b^{\ell + 1}$ greater than its mean, given $\T_n$, 
is at most $C_k\cdot b^{-(2\ell + 2)}e^{-c_k n}$.   Because 
$n = \lceil b^{-\delta}\rceil$, this is finite when summed 
over all possible $m$ and $r$, implying that all but finitely often
\begin{align*}
\sum_{\{u_i\} \in \binom{\TT_n}{k}} 
   \P_\TT\left[\rtt \conn_{p_c + b} u_1,\ldots,u_k\right]
& g(\TT(u_1),p_c + b)\cdots g(\TT(u_k),p_c + b) \\
& \leq g(p_c + b)^k\sum_{\{u_i\} \in \binom{\TT_n}{k}} 
\P_\TT\left[\rtt \conn_{p_c + b} u_1,\ldots,u_k\right] + b^{\ell+1} \, .
\end{align*}
	
By a similar argument, we obtain the lower bound 
\begin{align*}
\sum_{\{u_i\} \in \binom{\TT_n}{k}} 
   \P_\TT\left[\rtt \conn_{p_c + b} u_1,\ldots,u_k\right]
   & g(\TT(u_1),p_c + b)\cdots g(\TT(u_k),p_c + b) \\
& \geq g(p_c + a)^k\sum_{\{u_i\} \in \binom{\TT_n}{k}} 
   \P_\TT\left[\rtt \conn_{p_c + a} u_1,\ldots,u_k\right] - b^{\ell+1} \, .
\end{align*}
Letting $(*)$ denote the absolute value of the left-hand-side 
of~\eqref{eq:average-error}, we see that 
\begin{eqnarray*}
(*) & \leq & g(p_c + b)^k\sum_{\{u_i\} \in \binom{\TT_n}{k}} 
   \P_\TT\left[\rtt \conn_{p_c + b} u_1,\ldots,u_k\right] 
   - g(p_c + a)^k\sum_{\{u_i\} \in \binom{\TT_n}{k}} 
   \P_\TT\left[\rtt \conn_{p_c + a} u_1,\ldots,u_k\right]+ 2b^{\ell+1}   \\
& \leq & 2 (g(p_c + b)^k - g(p_c + a)^k) X_n^{(k)} + g(p_c + b)^k 
   \left(\P_\TT\left[\rtt \conn_{p_c + b} u_1,\ldots,u_k\right] 
   - \P_\TT\left[\rtt \conn_{p_c + a} u_1,\ldots,u_k\right]\right) 
   + 2b^{\ell + 1} \\
& \leq & 2 (g(p_c + b)^k - g(p_c + a)^k) X_n^{(k)} 
   + g(p_c + b)^k \frac{2 \cdot n\cdot k (b - a)}{p_c} X_n^{(k)} 
   + 2b^{\ell + 1} \, ,
\end{eqnarray*}
where the last inequality is via the Mean Value Theorem.
	
Dividing by $\ee^\ell$ and setting $C_k = 2k / p_c$, we have 
\begin{align*}
2 \frac{g(p_c + b)^k - g(p_c +a)^k}{\ee^\ell}&X_n^{(k)} 
   + C_k \cdot n\cdot g(p_c + b)^k 
   \frac{b - a}{\ee^\ell} X_{n}^{(k)}+2b(b/\ee)^\ell \\
& \leq  2\frac{b - a}{\ee^\ell} \cdot \frac{g(p_c+b)^k - g(p_c + a)^k}{b - a} 
   X_n^{(k)} + C_k\cdot n \cdot g(p_c + b)^k \frac{b - a}{\ee^\ell} 
   X_{n}^{(k)} +  2b (b/a)^\ell \\
& \leq  2k\frac{b - a}{\ee^\ell} \max_{x \in [p_c,1]}  
   g'(x) X_n^{(k)} + C_k \cdot n \cdot g(p_c + b)^k 
   \frac{b - a}{\ee^\ell} X_{n}^{(k)} + 2b \left ( \frac{b}{a} \right )^\ell 
\end{align*}
again by the Mean Value Theorem.  
	
By Theorem \ref{th:Y-jk-mart}$(c)$, $n^{-(k-1)} X_n^{(k)}$ 
converges as $n \to \infty$.  By definition of $b,a$ and $n$, 
$\frac{(b - a)n^{k}}{\ee^\ell} \to 0$ as $\ee \to 0$ for 
$\delta$ sufficiently small, thereby completing the proof except in 
the case when $\lceil a^{-\delta} \rceil \neq \lceil b^{-\delta} \rceil$.	
	
When $\lceil a^{-\delta} \rceil$ and $\lceil b^{-\delta} \rceil$ differ, 
we can split the interval $[a,b)$ into subintervals
$[a,c-\delta'),[c-\delta',c)$ and $[c,b)$, where $c \in (a,b)$ 
is the point where $\lceil x^{-\delta} \rceil$ drops.  Repeating 
the above argument for the first and third intervals, taking $\delta'$ 
sufficiently small, and exploiting continuity of the expression 
in~\eqref{eq:average-error} on $[a,c)$ provides us with desired 
asymptotic bounds for the middle interval, hence the proof is complete.		$\Cox$
\end{proof}

As a midway point in proving Theorem~\ref{th:g-expansion}, we obtain
an expansion for $g(\TT , p_c + \ee)$ that for a given $\ee$ is 
measurable with respect to $\T_{n(\ee)}$, where $n(\ee)$ grows like 
a small power of $\ee^{-1}$.  

\begin{lem}\label{lem:g-bonf}
Suppose $\E[Z^{(2\ell + 1)(1 + \beta)}] < \infty$ for some $\ell \geq 1$ 
and $\beta > 0$.  Define $n(\ee) := \lceil \ee^{-\delta} \rceil$. 
Then for $\delta > 0$ sufficiently small, we have $\GW$-a.s.\ the 
following expansion as $\ee \to 0^+$: 
$$g(\TT,p_c+\ee) = \sum_{i = 1}^\ell \left(\sum_{j = 1}^{i}   
   (-1)^{j+1} \sum_{d = j}^i p_c^d r_{j,d} X_{n(\ee)}^{(j,i-d)}\right) 
   \mu^i \ee^i + o(\ee^\ell).$$
\end{lem}

\begin{proof}
For each $j$ and $n$, define 
\begin{align*}
\widetilde{\Bon}_n^{(j)}(\ee) 
& := \sum_{\{v_i\} \in \binom{\TT_n}{j}}
   \P_\TT[\rtt \conn_p v_1,\ldots,v_j]g(\TT(v_1),p)\cdots g(\TT(v_j),p) \\
\text{ and } \Bon_n^{(j)}(\ee) &:= 
   \sum_{\{v_i\} \in \binom{\TT_n}{j}}
   \P_\TT[\rtt \conn_p v_1,\ldots,v_j]g(p)^j
\end{align*} 
where we write $p = p_c + \ee$.  Applying the Bonferroni inequalities 
to the event $\{ \rtt \conn_p \infty \} = \bigcup_{v \in T_n}
\{ \rtt \conn_p v \conn \infty \}$ yields
\begin{equation}\label{eq:bonf}
\sum_{i = 1}^{2j} (-1)^{i+1} \cdot \widetilde{\Bon}_{n(\ee)}^{(i)}(\ee)
   \leq  g(\TT,p_c + \ee) \leq \sum_{i = 1}^{2j\pm 1}(-1)^{i+1} \cdot  
   \widetilde{\Bon}_{n(\ee)}^{(i)}(\ee)
\end{equation} for each $j$, where the $\pm$ may be either a plus or minus.   
	
For sufficiently small $\delta > 0$, Proposition~\ref{pr:bonf-replace} 
allows us to replace each $\widetilde{\Bon}_{n(\ee)}^{(i)}(\ee)$ 
with $\Bon_{n(\ee)}^{(i)}(\ee)$, introduce an $o(\ee^{\ell})$ error term,
provided $\E[Z^{(2i - 1)(1 + \beta)}] < \infty$.  Moreover, we note 
$$\Bon_{n(\ee)}^{(i)}(\ee) = g(p_c + \ee)^i \sum_{\{v_r\} \in 
   \binom{\TT_{n(\ee)}}{i}} \P_\TT [\rtt \conn_{p_c+\ee} v_1,\ldots,v_i] 
   \leq C g(p_c + \ee)^i X_{n(\ee)}^{(i,0)} = o(\ee^{i-1}) \, .$$
The consant $C$ is introduced when we bound 
$(1 + \ee/p_c)^{|\TT(v_1,\ldots,v_i)|}$ from above by a constant $C$ for 
$\delta$ sufficiently small; the limit follows from 
Theorem~\ref{th:Y-jk-mart}$(c)$.  For each $j$, 
apply~\eqref{eq:bonf} to show 
\begin{equation}\label{eq:Bon-j-first}
g(\TT,p_c + \ee) = \sum_{j = 1}^\ell (-1)^{j+1} \Bon_n^{(j)}(\ee) 
   + o(\ee^\ell) \, .
\end{equation}
Now expand 
\begin{align}
\Bon_n^{(j)}(\ee) &= g(p_c + \ee)^j \sum_{\{v_i\} \in 
   \binom{\TT_n}{j}} (p_c + \ee)^{|\TT(v_1,\ldots,v_j)|} \nonumber \\
& = \left(\sum_{i = j}^\ell r_{j,i} \ee^i + o(\ee^\ell)\right) 
   \sum_{\{v_i\} \in \binom{\TT_n}{j}} 
   p_c^{|\TT(v_1,\ldots,v_j)|}(1 + \ee/p_c)^{|\TT(v_1,\ldots,v_j)|} 
   \nonumber \\
&=  \left(\sum_{i = j}^\ell r_{j,i} \ee^i + o(\ee^\ell)\right) 
   \sum_{\{v_i\} \in \binom{\TT_n}{j}} p_c^{|\TT(v_1,\ldots,v_j)|}
   \left(\sum_{i = 0}^\ell \binom{|\TT(v_1,\ldots,v_j)|}{i} 
   \frac{\ee^i}{p_c^i} + O(n^{\ell+1}\ee^{\ell+1})\right) \nonumber \\
& = \left(\sum_{i = j}^\ell r_{j,i} \ee^i + o(\ee^\ell)\right) 
   \left(\sum_{i = 0}^\ell X_n^{(j,i)} \frac{\ee^i}{p_c^i}  
   + o(\ee^\ell)\right) \nonumber \\
& = \sum_{i = j}^\ell \mu^i \ee^i \left(\sum_{d = j}^{i} p_c^d 
   r_{j,d} X_n^{(j,i-d)} \right) + o(\ee^\ell) \label{eq:Bon-j-second} \, .
\end{align}
Plugging this into~\eqref{eq:Bon-j-first} completes the Lemma. 
$\Cox$
\end{proof}

We are almost ready to prove Theorem~\ref{th:g-expansion}.  We have
dealt with the martingale part.  What remains is to get rid of the
predictable part.  The following combinatorial identity is the key
to making the predictable part disappear.

\begin{lem}\label{lem:constants-sum}
Fix $i \geq 1$ and suppose $\E[Z^{i+1}] < \infty$; then for each 
$a,b \leq i$ we have 
$$\sum_{d = 1}^i \sum_{j = 1}^i (-1)^{j-1} p_c^{d} r_{j,d} c_{j,a} 
   \binom{j}{b - d} = (-1)^{a + 1} p_c^b r_{a,b} \, .$$
\end{lem}

\begin{proof}
Begin as in the proof of Proposition \ref{pr:g-rec-rel} with the identity
$$\bigg[1 - \phi(1 - (p_c + \ee)g(p_c + \ee))\bigg]^a = g(p_c + \ee)^a \, .$$
The idea is to take Taylor expansions of both sides and equate coefficients 
of $\ee^b$; more technically, taking Taylor expansions of both sides up 
to terms of order $o(\ee^i)$ yield two polynomials in $\ee$ of degree $i$ 
whose difference is $o(\ee^i)$ thereby showing the two polynomials are equal.  
The coefficient $[\ee^b] g(p_c + \ee)^a$ of $\ee^b$ on the right-hand side 
is $r_{a,b}$, by definition.  On the left-hand-side, we write 
\begin{align*}
\bigg[1 - \phi(1 - (p_c + \ee)g(p_c + \ee))\bigg]^a 
& = \bigg[\sum_{k = 1}^i (-1)^{k+1}(1 + \ee/p_c)^k g(p_c + \ee)^k 
   p_c^k\frac{\phi^{(k)}(1)}{k!} + o(\ee^i)\bigg]^a \\
&=(-1)^a \sum_{j = 1}^i (-1)^{j}(1 + \ee/p_c)^j g(p_c + \ee)^j c_{j,a} 
   + o(\ee^i) \, .
\end{align*}
The coefficient of $\ee^b$ of $(1 + \ee/p_c)^j g(p_c + \ee)^j$ is 
\begin{align*}
[\ee^b] (1 + \ee/p_c)^j g(p_c + \ee)^j 
&= \sum_{ d = j}^b \left([\ee^d] g(p_c + \ee)^j\right) 
   \left([\ee^{b-d}](1 + \ee/p_c)^j \right) \\
& = \sum_{d = j}^b r_{j,d} \binom{j}{b-d} p_c^{-(b-d)} \, .
\end{align*}
Equating the coefficients of $\ee^b$ on both sides then gives 
$$(-1)^a \sum_{j = 1}^i (-1)^{j} c_{j,a} \sum_{d = j}^b r_{j,d} 
   \binom{j}{b - d} p_c^{-(b - d)} = r_{a,b} \, .$$
Multiplying by $p_c^b(-1)^{a+1}$ on both sides completes the proof.  
$\Cox$
\end{proof}

With Theorem~\ref{th:Y-jk-mart} and Lemma \ref{lem:constants-sum} in place, the limits of $M_n^{(i)}$ fall out easily.

\begin{lem}\label{lem:M_n-mart}
Suppose $\E[Z^{i+1}] < \infty$ for some $i$ and let $\beta > 0$ with 
$\E[Z^{i(1 + \beta)}] < \infty$.  Then 
\begin{enumerate}[$(a)$]
\item The sequence $(M_n^{(i)})_{n = 1}^\infty$ is a martingale 
with respect to the filtration $(\T_n)_{n=1}^\infty$. 
\item There exist positive constants $C, c$ depending only on 
$i, \beta$ and the progeny distribution so that 
$\Vert M_{n + 1}^{(i)} - M_n^{(i)}\Vert_{L^{1+\beta}} \leq C e^{- c n}.$
\item There exists a random variable $M^{(i)}$ so that $M_n^{(i)} \to M^{(i)}$ 
both almost surely and in $L^{1+\beta}$.
\end{enumerate}   
\end{lem}

\begin{proof}
Note first that $(c)$ follows from $(a)$ and $(b)$ by the triangle 
inequality together with the $L^p$ martingale convergence theorem.  
Parts $(a)$ and $(b)$ are proved simultaneously.  Write 
\begin{align}
\mu^{-i} \left ( M_{n+1}^{(i)} - M_n^{(i)} \right ) &= \sum_{j = 1}^i (-1)^{j+1}
   \sum_{d = j}^i p_c^d r_{j,d} \left(X_{n+1}^{(j,i-d)} - X_n^{(j,i-d)}\right) 
   \nonumber \\
& = \sum_{j = 1}^i (-1)^{j+1}\sum_{d = j}^i p_c^d r_{j,d} 
   \left(\Delta Y_{n+1}^{(j,i-d)} + \sum_{a = 1}^j c_{j,a} 
   \sum_{b = 0}^{i -d} \binom{j}{i-d-b} X_n^{(a,b)} - X_n^{(j,i-d)}\right) 
   \nonumber \\
& = \sum_{j = 1}^i \sum_{d =j }^i (-1)^{j+1} p_c^d r_{j,d} 
   \Delta Y_{n+1}^{(j,i-d)} \nonumber \\
& \qquad + \sum_{j = 1}^i \sum_{d = j}^i (-1)^{j+1}p_c^d r_{j,d} 
   \left(\sum_{a = 1}^j \sum_{b = 0}^{i - d} c_{j,a} 
   \binom{j}{i-d-b} X_n^{(a,b)} - X_n^{(j,i-d)}\right)\,. \label{eq:M_n-diff}
\end{align}
By Theorem \ref{th:Y-jk-mart}, we have that 
$\Delta Y_{n+1}^{(j,i-d)}$ is exponentially small in $L^{1+\beta}$.  
This means that we simply need to handle the second sum 
in~\eqref{eq:M_n-diff}.  We claim that it is identically equal to zero.   
This is equivalent to the claim that 
\begin{equation} 
\sum_{j = 1}^i \sum_{d = j}^i \sum_{a = 1}^j 
   \sum_{b = 0}^{i - d} (-1)^{j+1}p_c^d r_{j,d} 
   c_{j,a} \binom{j}{i-d-b} X_n^{(a,b)} 
= \sum_{a = 1}^i \sum_{b = a}^i (-1)^{a + 1}p_c^b r_{a,b}X_n^{(a,i-b)} 
   \, .  \label{eq:need} 
\end{equation}
	
To prove this, we rearrange the sums in the left-hand-side 
of~\eqref{eq:need}.  In order to handle the limits of each sum, 
we recall that $c_{j,a} = 0$ for $j < a$ and $r_{j,d} = 0$ for $d < j$.  
Relabeling and swapping gives 
\begin{eqnarray*}
\sum_{j = 1}^i \sum_{d = j}^i \sum_{a = 1}^j \sum_{b = 0}^{i - d} 
   (-1)^{j+1}p_c^d r_{j,d} c_{j,a} \binom{j}{i-d-b} X_n^{(a,b)}
& = & \sum_{j = 1}^i \sum_{d = j}^i \sum_{a = 1}^j \sum_{b = d}^{i} 
   (-1)^{j+1}p_c^d r_{j,d} c_{j,a} \binom{j}{b-d} X_n^{(a,i-b)} \\
& = & \sum_{a = 1}^i \sum_{b = a}^i  X_n^{(a,i-b)}
   \left( \sum_{d = 1}^i \sum_{j = 1}^i (-1)^{j-1} 
   p_c^d r_{j,d} c_{j,a} \binom{j}{b-d} \right) \, .
\end{eqnarray*}
Lemma \ref{lem:constants-sum} shows that the term in parentheses 
is equal to $(-1)^{a+1}p_c^b r_{a,b}$, thereby showing~\eqref{eq:need}. 
$\Cox$
\end{proof}

{\sc Proof of Theorem~\ref{th:g-expansion}:}  
Apply Lemma~\ref{lem:g-bonf} to obtain some $\delta > 0$ sufficiently 
small so that 
\begin{equation}\label{eq:g-M-n-exp}
g(\TT,p_c + \ee) = \sum_{i = 1}^\ell M_n^{(i)} \ee^i + o(\ee^\ell)
\end{equation} 
with $n = \lceil \ee^{-\delta}\rceil$.  The exponential convergence of 
$M_n^{(i)}$ from Lemma \ref{lem:M_n-mart} together with Markov's 
inequality and Borel-Cantelli shows that 
$$|M_n^{(i)} - M^{(i)}|n^{N} \to 0$$ 
almost surely for any fixed $N > 0$.  Because 
$n = \lceil \ee^{-\delta}\rceil$ implies $n^{-N} = o(\ee^\ell)$ 
for $N$ sufficiently large,~\eqref{eq:g-M-n-exp} can be simplified to 
$$g(\TT,p_c + \ee) = \sum_{i = 1}^\ell M^{(i)} \ee^i + o(\ee^\ell)\,.$$

It remains only to show that $\E M^{(i)} = r_i$.  Because 
$M_n^{(i)}$ converges in $L^{1+\beta}$, it also converges in $L^1$, 
implying $\E[M^{(i)}] = \E[M_1^{(i)}]$.  Noting that 
$\E[X_1^{(j,k)}] = \binom{j}{k}c_{j,1}$, we use 
Lemma~\ref{lem:constants-sum} with $a = 1$ and $b = i$ in the
penultimate line to obtain
\begin{eqnarray*}
p_c^i \E[M_1^{(i)}] + \sum_{j = 1}^i \sum_{d = j}^i (-1)^{j+1} 
   p_c^d r_{j,d} \E[X_1^{(j,i-d)}] 
& = & \sum_{j = 1}^i \sum_{d = j}^i (-1)^{j+1} p_c^d r_{j,d} 
   \binom{j}{i -d}c_{j,i} \\
& = & (-1)^{1 + 1}p_c^i r_{1,i} \\
& = & p_c^i r_i \, .
\end{eqnarray*}
$\Cox$

\section{Regularity on the Supercritical Region} \label{sec:cont diff}
	
In this section we prove Russo-type formulas expressing the 
derivates of $g(T,p)$ as expectations of quantities measuring
the number of pivotal bonds.  The first and simplest of these 
is Theorem~\ref{th:c1}, expressing $g'(T,p)$ as the 
expected number of pivotal bonds multiplied by $p^{-1}$.  
In Section~\ref{sec:c3} we define some combinatorial gadgets to 
express more general expectations (Definitions~\ref{def:monomials}) 
and show that these compute successive derivatives 
(Proposition~\ref{pr:JRformofderiv}).  
In Section~\ref{sec:cont-at-pc}, explicit estimates on these
expectations are given in Proposition~\ref{pr:k+1/k}, which under
suitable moment conditions lead to continuity of the first
$k$ derivatives at $p_c^+$, which is Theorem~\ref{th:contderivpc}.

\subsection{Continuous differentiability on $(p_c,1)$}\label{sec:g'}

Given $T$ and $p$, let $T_p = T_p (\omega)$ denote the tree obtained
from the $p$-percolation cluster at the root by removing all vertices
$v$ not connected to infinity in $T(v)$.  Formally, $v \in T_p$
if and only if $\rtt \conn_{T,p} v$ and $v \conn_{T(v),p} \infty$.
On the survival event $H_T(p)$ let $B_p$ denote the first node at which
$T_p$ branches.  Formally, the event $\{ B_p = v \}$ is the intersection
of three events $\open (v), \Nobranch (v)$ and $\branch (v)$ 
where $\open (v)$ is the event $\rtt \conn_{T,p} v$ of the path 
from the root to $v$ being open, $\Nobranch (v)$ is the event that 
for each ancestor $w < v$, no child of $w$ other than the one that is 
an ancestor of $v$ is in $T_p$, and $\branch (v)$ is the event that 
$v$ has at least two children in $T_p$. We call $|B_p|$ the 
{\em branching depth}.  The main result of this
subsection is the following.

\begin{thm} \label{th:c1}
The derivative of the quenched survival function is given by 
$g'(T,p) = p^{-1} \E_T |B_p|$, which is finite and continuous 
on $(p_c,1)$.
\end{thm}

We begin with two annealed results.  Although it may be obvious, we 
point out that the notation $\TT_p$ means to take the random tree $\TT$ 
defined by the $\deg_v$ variables and apply the random map 
$\TT\mapsto\TT_p$ defined by the $U_v$ variables (see 
Section~\ref{sec:bpintro} for relevant definitions).  
Recall that $g(p)$ denotes the annealed survival function.
The  following fact is elementary and follows from taking the thinned
o.g.f.\ $\psi (z,s) := \phi (1 - s (1-z))$, setting $s = p g(p)$ 
to obtain the survivor tree, then conditioning on being non-empty, i.e.,
$(\psi - \psi (0)) / (1 - \psi (0))$:

\begin{pr} \label{pr:T_p}
For any $p > p_c$ define an offspring generating function 
\begin{equation} \label{eq:phi_p}
\phi_p (z) := \frac{\phi (1 - p g(p) (1-z)) - \phi (1 - p g(p))}
   {g(p)} \, .
\end{equation}
Then the conditional distribution of $\TT_p$ given $H(p)$ is 
Galton-Watson with offspring generating function $\phi_p$, which
we will denote $\GW_p$.
$\Cox$
\end{pr}

\begin{lem}[annealed branching depth has exponential moments]
\label{lem:exp moment}
Let
\begin{equation} \label{eq:A_p}
A_p = A_p (\phi) := \phi_p' (0)
\end{equation}
denote the probability under $\GW_p$ that the root has precisely 
one child.  Suppose $r > 0$ and $p > p_c$ satisfy $(1+r) A_p < 1$.  
Then $\E (1+r)^{|B_p|} < \infty$.
\end{lem}

\noindent{\sc Proof:}
The result is equivalent to finiteness of $\sum_{n=1}^\infty 
(1+r)^n \P (|B_p| \geq n)$.  Proposition ~\ref{pr:T_p} implies that
$\P (|B_p| \geq n) = A_p^n$, showing $\E (1+r)^{|B_p|}$ to
be the sum of a convergent geometric series.
$\Cox$

Next we recast the $p$-indexed stochastic process 
$\{ T_p : p \in [0,1] \}$ as a Markov chain.
Define a filtration $\{ \G_p : 0 \leq p \leq 1 \}$ by 
$\G_p = \sigma (\T, \{ U_e \vee p \})$.  
Clearly if $p > p'$ then $\G_p \subseteq \G_{p'}$, thus $\{ \G_p \}$
is a filtration when $p$ decreases from~1 to~0.  Informally, 
$\G_p$ knows the tree, knows whether each edge $e$ is open 
at ``time'' $p$, and if not, ``remembers'' the time $U(e)$
when $e$ closed. 

\begin{lem} \label{lem:p markov}
Fix any tree $T$.  The edge processes $\{ \one_{U(e) \leq p} \}$
are independent left-continuous two-state continuous time Markov chains. 
They have initial state~1 when $p=1$ and terminal state~0 when $p=0$,
and they jump from~1 to~0 at rate $p^{-1}$.  The process $\{ T_p \}$ 
is a function of these and is also Markovian on $\{ \G_p \}$.
\end{lem}

\noindent{\sc Proof:} Independence and the Markov property for
$\{ \one_{U(E) \leq p} \}$ are immediate.  The jump rate is the
limit of $\ee^{-1} \P (U(e) \in (p-\ee,p) | U(e) < p)$ which is $p^{-1}$.
The Markov property for $T_p$ on $\{ \G_p \}$ is immediate because it is 
a function of Markov chains on $\{ \G_p \}$.
$\Cox$

\noindent
Next we define the quantity $\beta$ as 
$\beta:=\text{inf}\{p:T_p\ \text{is infinite}\}$.  Thus 
$g(T,p)=\P_T(\beta\leq p)$ and $g'(T,p)$ is the density, if it exists, 
of the $\P_T$-law of $\beta$.  Before proceeding to the proof of 
Theorem \ref{th:c1}, we will need to establish one additional lemma.

\begin{lem} \label{lem:splitH}
With probability $1$, at $p=\beta$ the root of $T_p$ is connected to 
infinity, $|B_p|<\infty$ (i.e. $T_p$ does branch somewhere for $p=\beta$),
and there is a vertex $v\leq B_p$ with $U_v=\beta$.  Consequently, the 
event $H(p)$ is, up to measure $0$, a disjoint union of the events 
$\{B_{\beta}=v\}\cap\{\beta\leq p\}$.
\end{lem}

\noindent{\sc Proof:} As $p$ decreases from $1$ to $0$, the vertex $B_p$ 
can change only by jumping to a descendant or to zero.  Either zero is 
reached after finitely many jumps, in which case $B_{\beta}$ is the value 
just before jumping to zero, or there is a countable sequence of jumps.  
The decreasing limit of the jump times must be at least $p_c$ because 
below $p_c$ the value of $B_p$ is always zero.  The set of infinite paths,
 also known as $\partial T$, is compact, which means a decreasing 
intersection of closed subsets of $\partial T$ is non-empty.  It follows 
that if $\beta$ is the decreasing limit of jump times for $B_p$ then 
$H(\beta)$ occurs, that is $\rtt \conn\infty$ for $p=\beta$.  Because 
$g(p_c)=0$, we conclude the probability of a countable sequence of jump 
times decreasing to $p_c$ is zero.  To prove the lemma therefore, it 
suffices to rule out a sequence of jump times decreasing to some value 
$y>p_c$.  Now for the annealed process, $\{B_p:p>p_c\}$ is a 
time-inhomogeneous Markov chain, the jump rate depends only on $p$ and 
not $|B_p|$, and the jump rate is bounded for $p\in [y,1]$.  It follows 
that for almost surely every $T$, the probability of infinitely many 
jumps in $[y,1]$ is zero for any $y>p_c$.
$\Cox$

\noindent{\sc Proof of Theorem}~\ref{th:c1}:
By Lemma~\ref{lem:splitH} $H_T(p)$ is equal to the union of the disjoint 
events $\{B_{\beta}=v\}\cap H_T(p)$.  On
$\{ B_{\beta} = v \}$ the indicator $\one_{H(p)}$ jumps to zero precisely 
when $\open(v)$ does so, which occurs at rate $p^{-1} |v|$. 
Because all jumps have the same sign, it now follows that 
$$\frac{d}{dp} \, g(T , p) 
   = \frac{1}{p} \sum_{v \in T} |v| \P (B_p = v)
   = \frac{1}{p} \E |B_p| \, ,$$
which may be $+\infty$.  Summing by parts, we also have
\begin{equation} \label{eq:tail sum}
\frac{d}{dp} \, g(T,p) = \frac{1}{p} \sum_{v\neq\rtt} \P (B_p \geq v)  
\end{equation}
where $B_p \geq v$ denotes $B_p = w$ for some descendant $w$ of $v$.

To see that this is finite and continuous on $(p_c,1)$, consider any
$p' > p_c$ and $r > 0$ with $(1+r) p' A_{p'} < 1$.  For any $p \in 
(p' , (1+r) p')$ we have 
$$\P_T(B_p \geq v) = \P_T(\open (v)) \P (\Nobranch (v,p))g(T(v),p) 
   \leq (1+r)^{|v|} (p')^{|v|} \P_T(\Nobranch (v,p')).$$

Taking the expectation of the expression on the right and multiplying by 
$g(p')$ we observe that 
\begin{align*}g(p')\E\Bigg[\sum_{v\in\TT}(1+r)^{|v|}(p')^{|v|}
\P_{\TT}(\Nobranch(v,p'))\Bigg]&=\E\Bigg[\sum_{v\in\TT}(1+r)^{|v|}
(p')^{|v|}\P_{\TT}(\Nobranch(v,p'))g(\TT(v),p')\Bigg] \\ 
&=\E\Bigg[\sum_{n=0}^{\infty} (1+r)^n\P(B_{p'}\geq n)\Bigg]<
\infty \end{align*}
where the last inequality follows from Lemma~\ref{lem:exp moment}.  This 
now implies that for $\GW$-almost every $T$, the right-hand-side of
 ~\eqref{eq:tail sum} 
converges uniformly for $p \in (p' , (1+r) p')$, thus implying 
continuity on this interval.  Covering $(p_c,1)$
by countably many intervals of the form $(p' , (1+r) p')$, the theorem
follows by countable additivity.
$\Cox$

\subsection{Smoothness of $g$ on the Supercritical Region}\label{sec:c3}
	
Building on the results from the previous subsection, we establish the
main result concerning the behavior of the quenched survival function 
in the supercritical region.
\begin{thm} \label{th:cinf}
For $\GW$-a.e. $T$, $g(T,p) \in C^{\infty}((p_c,1))$.
\end{thm}
	
In order to prove this result, we define quantities generalizing
the quantity $\E_T |B_p|$ and show that the derivative of a function
in this class remains in the class.  We begin by presenting several 
definitions.  Throughout the remainder of the paper, our trees are 
rooted and {\em ordered}, meaning that the children of each vertex 
have a specified order (usually referred to as left-to-right rather 
than smallest-to-greatest) and isomorphisms between trees are understood 
to preserve the root and the orderings.  

\begin{defns} \label{def:monomials}
~~\\[-6ex]
\begin{enumerate}[$(i)$]
\item {\bf Collapsed trees.}
Say the tree $\V$ is a collapsed tree if no vertex of
$\V$ except possibly the root has precisely one child.
\item {\bf Initial subtree.}  The tree $\Ttilde$ is said to be
an initial subtree of $T$ if it has the same root, and if for
every vertex $v \in \Ttilde \subseteq T$, either all children of
$v$ are in $\Ttilde$ or no children of $v$ are in $\Ttilde$, with
the added proviso that if $v$ has only one child in $T$ then it must 
also be in $\Ttilde$.
\item {\bf The collapsing map $\Phi$.}  
For any ordered tree $T$, let $\Phi(T)$ denote the isomorphism
class of ordered trees obtained by collapsing to a single edge
any path along which each vertex except possibly the last has 
only one child in $T$ (see figure below).
\item {\bf Notations $T(\V)$ and $\V \preceq T$.}  It follows from the
above definitions that any collapsed tree $\V$ is isomorphic to 
$\Phi (\Ttilde)$ for at most one initial tree $\Ttilde \subseteq T$.  
If there is one, we say that
$\V \preceq T$ and denote this subtree by $T(\V)$.  We will normally
use this for $T = T_p$.  For example, when $\V$ is the tree with one
edge then $\V \preceq T_p$ if and only if $T_p$ has precisely one
child of the root, in which case $T_p (\V)$ is the path from the root 
to $B_p$.
\item {\bf The embedding map $\Phiinv$.}  
If $e$ is an edge of $\V$ and $\V \preceq T$, let $\embed (e)$
denote the path in $T(\V)$ that collapses to the edge carried
to $e$ in the above isomorphism.  For a vertex $v \in \V$ 
let $\embed (v)$ denote the last vertex in the path $\embed (e)$
where $e$ is the edge between $e$ and its parent; if $v$ is the
root of $\V$ then by convention $\embed (v)$ is the root of $T$.
\item {\bf Edge weights.}  
If $\V \preceq T$ and $e \in E(\V)$, define $d(e) = d_{T,\V} (e)$ 
to be the length of the path $\embed (e)$.
\item {\bf Monomials.}  A monomial in (the edge weights of) a collapsed
tree $\V$ is a set of nonnegative integers $\{ F(e) : e \in E(\V) \}$ 
indexed by the edges of $\V$, identified with the product 
\begin{equation}\label{defmonomt}
\langle T, \V, F \rangle := 
   \begin{cases}
	\prod_{e \in E(\V)} d (e)^{F(e)} 
	& \text{ if } \V \preceq T \, , \\ 
   0 & \text{ otherwise. }
\end{cases} \,. \end{equation}  
A monomial $F$ is only defined in reference to a weighted collapsed 
tree $\V$.  As an example, if $T$ is as in Figure~\ref{Defining-Phi}
with $\V = \Phi(T)$, and if we take $F (e) = 1$ for the three edges 
down the left, $F (e) = 3$ for the rightmost edge and 
$F (e) = 0$ for the other four edges, then 
$\langle T, \V, F \rangle = 2 \cdot 3 \cdot 2 \cdot 2^3$.  
\end{enumerate}
\end{defns}

\begin{figure}[!ht]
	\centering
	\includegraphics{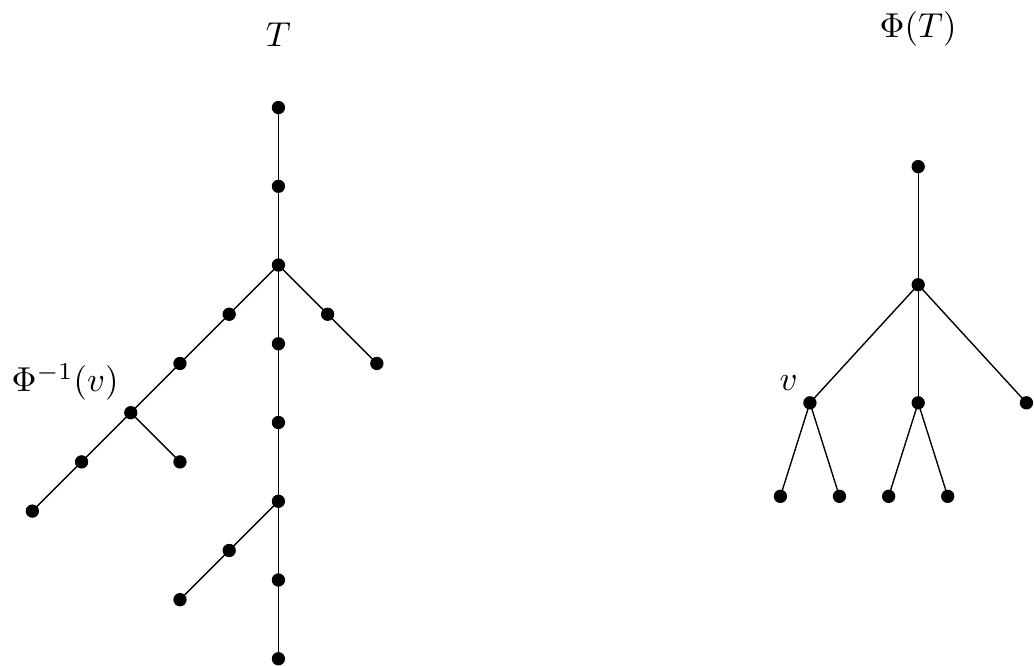}
	
	\caption{Illustrations showing how $\Phi$ acts on a tree $T$.}
	\label{Defining-Phi}
	
\end{figure}

\begin{defns}[monomial expectations]
Given $T, p$, a positive real number $r$, a finite collapsed tree 
$\V$, and a monomial $F$, define functions 
$\Trp = \Trp (T,r,\V,p)$ and $\TFV = \TFV (T,F,\V,p)$ by
\begin{eqnarray}
\Trp & := & \E_T \left[(1+r)^{|E (T_p (\V))|} \one_{\V \preceq T_p} \right]
   \label{eq:def E} \\
\TFV & := & \E_T \left[\langle T_p, \V, F \rangle \right]\,.\label{eq:def V}
\end{eqnarray}
\end{defns}
For example, if $\V_1$ is the tree with a single edge $e$ and 
$F_1 (e) = 1$, then $\langle T_p, \V_1 , F_1 \rangle = |B_p|$ and the 
conclusion of Theorem~\ref{th:c1} is that for $p > p_c$,
\begin{equation} \label{eq:N=1}
\frac{d}{dp} \, g(T,p) = \frac{1}{p}
\E_T |B_p| = \frac{1}{p}\TFV (T, F_1 , \V_1 , p) \, .
\end{equation}
The main result of this section, from which Theorem~\ref{th:cinf}
follows without too much further work, is the following representation.

\begin{pr}\label{pr:JRformofderiv}
Let $\V$ be a collapsed tree and let $F$ be a monomial in
the variables $d_{T,\V} (e)$.  Then there exists a collection 
of collapsed trees $\V_1,\dots,\V_m$ for some $m \geq 1$ and 
monomials $F_1, \ldots , F_m$, given explicitly 
in~\eqref{eq:explicit} below, such that 
\begin{equation}\label{rderivexpfm}
\frac{d}{dp} \E_T \langle T_p, \V, F \rangle = 
   \frac{1}{p} \sum_{i=1}^m \E_T \langle T_p, \V_i , F_i \rangle
\end{equation}
on $(p_c,1)$ and is finite and continuous on $(p_c , 1)$ for 
$\GW$-a.e. tree $T$.  Furthermore, each monomial $F_i$
on the right-hand side of~\eqref{rderivexpfm} satisfies $\deg (F_i)
= 1 + \deg (F)$ and each of the edge sets $E(\V_i)$ satisfies 
$|E(\V_i)| \leq 2 + |E(\V)|$.
\end{pr}

The content of this result is twofold: that the derivative of a 
random variable expressible in the form \eqref{rderivexpfm} is also 
expressible in this form, and that
all derivatives are continuous on $(p_c , 1)$.  The proof 
of~\eqref{rderivexpfm} takes up some space due to the bulky 
sums involved.  We begin with two finiteness results
that are the analogues for $\Trp$ and $\TFV$ of the exponential moments
of $B_p$ proved in Proposition~\ref{lem:exp moment}.  Recall the
notation $A_p$ for the probability of exactly one child of the root
having an infinite descendant tree given that at least one does.

\begin{lem} \label{lem:R}
If $\V$ is a collapsed tree, $p \in (p_c,1)$, and $r > 0$ with
$(1+r) A_p < 1$, then
$\E \Trp (\TT,r,\V,p) < \infty \, .$
\end{lem}

\begin{proof} 
Let $Z_p$ denote a variable distributed as the first generation 
of $\TT_p$ conditioned on $H(p)$, in other words, having PGF $\phi_p$.
The annealed collapsed tree $\Phi (\TT_p)$, with weights, has a simple 
description.  The weights $d(\parent(v),v)$ of the edges will be IID 
geometric variables with mean $1 / (1-A_p)$ because the length of the 
descendant path needed to reach a node with at least two children is geometric 
with success probability $1 - A_p$.  Let $G$ denote a geometric variable 
with this common distribution.  The number of children of each node other than 
the root will have offspring distribution $(Z_p | Z_p \geq 2)$, whereas
the offspring distribution at the root will be $Z_p$ because
the root is allowed to have only one child in $T_p$.

The cases where $\V$ is empty or a single vertex being trivial, we
assume that $\V$ has at least one edge.
The event $\V \preceq \TT_p$ is the intersection of the event $H(p)$
with the events $\deg_{\TT_p} (\embed (v)) = \deg_{\V} (v)$ as $v$ varies 
over the interior vertices of $\V$ including $\rtt$.  The branching structure 
makes these events (when conditioned on $H(p)$) independent with probabilities
$\P (Z_p = \deg (v) | Z_p \geq 2)$, except at the root where
one simply has $\P (Z_p=\deg(\rtt))$ (with $\deg (v)$ representing the number 
of children of $v$).  It follows from this and
the IID geometric edge weights that 
\begin{equation} \label{eq:GW rep}
\E \Trp (\TT, r, \V, p) = g(p) \left [ \E (1+r)^G \right ]^{|E(\V)|} 
   \P (Z_p = \deg (\rtt))
   \prod_v \P (Z_p = \deg (v) - 1 | Z_p \geq 2) \, ,
\end{equation}
where the product is over vertices $v$ in $\V \setminus 
(\partial \V \cup \{ \rtt \})$.
Finiteness of $\E \Trp (\TT, r, \V, p)$ then follows from finiteness of
$\E (1+r)^G$, which was Lemma~\ref{lem:exp moment}. $\Cox$
\end{proof}

Before proceeding to give the proof of Proposition~\ref{pr:JRformofderiv},
we present one final result establishing finiteness and continuity of 
quenched monomial expectations.

\begin{lem}\label{lem:finexpmon}
If $\V$ is a collapsed tree and $F$ is a monomial, then for 
$\GW$-almost every $T$, the quantity $\TFV(T, F, \V, p)$ is 
finite and varies continuously as a function of $p$ on $(p_c,1)$. 
\end{lem}

\noindent{\sc Proof:}  
For finiteness and continuity, as well as for computing the explicit 
representation in~\eqref{eq:explicit} below, we need to decompose the 
monomial expectation according to the identity of the subtree 
$T_p (\V)$.  Observing that $\langle T_p, \V, F \rangle$ is constant 
on the event $\{ T_p (\V) = \Ttilde \}$, we denote this common value 
by $F(\V , \Ttilde)$.  Generalizing the decomposition of $H(p)$ in 
Section~\ref{sec:g'}, the event $\{T_p (\V) = \Ttilde \}$ is equal 
to the intersection of three independent events, $\open, \Nobranch$ 
and $\leafbranch$, depending on $T, \V$ and the parameters $\Ttilde$ 
and $p$, where:
\begin{quote}
$\open = \bigcap_{v \in \Ttilde} \open (v)$ is the event that 
all of $\Ttilde$ is open under $p$-percolation; \\

$\Nobranch = \bigcap_{v \in \V \setminus \partial \V} \Nobranch (v)$ 
is the event that no interior vertex $v \in \Ttilde$, has a child
not in $\Ttilde$; \\

$\leafbranch = \bigcap_{v \in \partial \B} \branch (v)$ is the event 
that the $p$-percolation cluster branches at every leaf of $\Ttilde$.
\end{quote}

Letting $\subt(T,\V)$ denote the set of subtrees $\Ttilde \subseteq T$ 
for which $\Ttilde$ has the same root as $T$, $\Phi(\Ttilde)$ is 
isomorphic to $\V$, and which satisfy the property that if $v\in\Ttilde$
has only one child in $T$ then that child is also in $\Ttilde$, we may write
$$\langle T_p, \V, F \rangle = \sum_{\Ttilde \in \subt (T,\V)} F(\V, \Ttilde)
   \one_{T_p (\V) = \Ttilde} \,$$
(note that $\Ttilde$ does not need to be an initial subtree of $T$, since
we are not making the assumption that all vertices in $T$ are either in 
$\Ttilde$ or are descendants of vertices in $\Ttilde$).  Taking expectations
now yields $$\TFV (T, F, \V, p) =  
\sum_{\Ttilde \in \subt (T,\V)} \TFV |_{\Ttilde}$$
where
\begin{eqnarray} \label{eq:D sum}
\TFV |_{\Ttilde} & := & \P (T_p (\V) = \Ttilde) F(\V , \Ttilde)  \\
& = & \P_T (\open) \P_T (\Nobranch) \P_T (\leafbranch) F(\V , \Ttilde) 
   \nonumber \\
& = & 
   p^{|E(\Ttilde)|}  
   \prod\limits_{\substack{w \notin \Ttilde \\ \parent(w) \in \Ttilde
   \setminus \partial \Ttilde}} (1 - p g(T(w),p))  
   \prod_{v \in \partial \Ttilde} g_2 (T(v) , p) F(\V , \Ttilde) 
   \nonumber \,
\end{eqnarray}

with $g_2(T,p)$ denoting the probability $T_p$ branches at the root.  
Each summand $\TFV |_{\Ttilde}$ in~\eqref{eq:D sum} is continuous in 
$p$, so it suffices to show that for $\GW$-almost every $T$, 
the sum converges uniformly on any compact subinterval 
$[a,b] \subseteq (p_c , 1)$.  For any $r > 0$ there are 
only finitely many $\Ttilde \in \subt(T,\V)$ for which 
$F(\V , \Ttilde) > (1+r)^{|E(\Ttilde)|}$.  We may therefore choose 
$C_r$ such that $F (\V , \Ttilde) \leq C_r (1+r)^{|E(\Ttilde)|}$
for all $\Ttilde$.  The result now follows similarly to as in the proof of
Theorem~\ref{th:c1}.  Cover $(p_c , 1)$ by countably many
intervals $(p' , (1+r)p')$ on which 
\begin{align}\label{eq:DTB}\TFV |_{\Ttilde} &\leq 
\Big(C_r (1+r)^{|E(\Ttilde)|}\Big)(p')^{|E(\Ttilde)|} 
(1+r)^{|E(\Ttilde)|}\prod\limits_{\substack{w \notin \Ttilde \\ 
\parent(w) \in \Ttilde\setminus \partial \Ttilde}} 
(1 - p' g(T(w),p')) \\ &=C_r(p')^{|E(\Ttilde)|} (1+2r+r^2)^{|E(\Ttilde)|}
\prod\limits_{\substack{w \notin \Ttilde \\ \parent(w) \in \Ttilde
\setminus \partial \Ttilde}} (1 - p' g(T(w),p')). \nonumber
\end{align}
Next note that \begin{align*}\E\Trp(\TT,2r+r^2,\V,p')
&=\E\Bigg[\sum(p')^{|E(\Ttilde)|}(1+2r+r^2)^{|E(\Ttilde)|}
\prod\limits_{\substack{w \notin \Ttilde \\ \parent(w) \in \Ttilde
\setminus \partial \Ttilde}} (1 - p' g(T(w),p'))  
\prod_{v \in \partial \Ttilde} g_2 (T(v) , p')\Bigg] \\ 
&=\Big(g_2 (p')\Big)^{|\partial\V|}
\E\Bigg[\sum(p')^{|E(\Ttilde)|}(1+2r+r^2)^{|E(\Ttilde)|}
\prod\limits_{\substack{w \notin \Ttilde \\ \parent(w) \in \Ttilde
\setminus \partial \Ttilde}} (1 - p' g(T(w),p'))\Bigg]\end{align*}
where the two sums are taken over all $\Ttilde\in\subt(\TT,\V)$.  
By Lemma~\ref{lem:R}, we know that if $(1+2r+r^2)A_{p'}<1$, then the 
above expectations are all finite.  It then follows that the sum over 
$\subt(T,\V)$ of the expression on the bottom in \eqref{eq:DTB} is 
finite for $\GW$-almost every $T$, thus implying uniform convergence 
of the sum of the  $\TFV|_{\Ttilde}$ expressions on $(p',(1+r)p')$.  
Now covering $(p_c,1)$ by a countable collection of such intervals, 
we see that $\TFV(T, F, \V, p)$ is both finite and continuous 
$\GW$-almost surely, thus completing the proof.
$\Cox$

\noindent{\sc Proof of Proposition}~\ref{pr:JRformofderiv}: ~~\\[2ex]
\ul{Step 1: Notation to state precise conclusion.}
Given a finite collapsed tree $\V$ and a monomial $F$, 
we define several perturbations of $\V$ and corresponding 
perturbations of $F$ and $\TFV$.
\begin{enumerate}[$(i)$]
\item Let $v$ be an interior vertex of $\V$, for which $L\geq 0$ of 
its children are also interior vertices.  Define $\V_i(v+1)$ 
(for $0\leq i\leq L$) to be $\V$ with a child $v_*$ added to $v$ via 
an edge $e_*$ placed anywhere between the $i$th and $(i+1)$th of the 
$L$ children of $v$ that are interior vertices (where we are counting 
from left to right).  Define the monomial $F^{(i)}_{v+1}$ by 
$F^{(i)}_{v+1} (e_*) = 1$ and $F^{(i)}_{v+1} (e) = F(e)$
for $e \neq e_*$.  Define $$\TFV_{v+1} (T, F, \V, p)  
:= \sum_{i=1}^L\TFV (T , F^{(i)}_{v+1} , \V_i(v+1) , p) \, .$$
\item
If $v$ is a vertex in $\partial \V$, i.e. a leaf, let $\V (v+2)$ 
denote the result of adding two children $v_0$ and $v_1$ to $v$ connected
by edges $e_0$ and $e_1$ respectively.  For $j = 0, 1$, define a monomial
$F_j$ on the edges of $\V(v+2)$ by $F_j = F$ on edges of $\V$, 
$F_j = 1$ on $e_j$ and $F_j = 0$ on $e_{1-j}$.  Define 
$$\TFV_{v+2} (T, F, \V , p) := 
   \TFV (T, F_0 , \V(v+2), p) + \TFV (T, F_1, \V(v+2) , p) 
   \, .$$
\item If $e$ is any edge in $E(\V)$, define two collapsed trees
$\V(e;L)$ and $\V(e;R)$ by subdividing $e$ at a new midpoint $m$
into edges $e_0$ and $e_1$, then adding a new child of $m$ with an edge
$e_*$ going to the left of $e_1$ in $\V(e;L)$ and to the right of $e_1$ 
in $\V(e;R)$.  
For $j = 0, 1$, define monomials $F_{j,L}$ on the edges of $\V(e;L)$
by $F_{j,L} (e') = F(e')$ for $e' \in E(\V)$ not equal to $e$, 
$F_{j,L} (e_j) = 1$, $F_{j,L} (e_{1-j}) = 0$, and $F_{j,L} (e_*) = 1$.
Define $F_{j,R}$ analogously on the edges of $\V(e;R)$ and define
$$\TFV_{e+2} (T, F, \V, p) := \sum_{j=0}^1 \sum_{A \in \{L,R\}} 
   \TFV (T, F_{j,A} , \V(e;A), p) \, .$$
\item Finally, for any edge $e \in E(\V)$, we define $F_e$ by
$F_e (e') = F(e')$ when $e' \neq e$ and $F_e (e) = F(e) + 1$.
In other words, the power of $d_T (e)$ in the monomial is bumped up
by one and nothing else changes.  Define
$$\TFV_{e} (T, F, \V, p) := \TFV (T, F_e, \V, p) \, .$$
\end{enumerate}

The explicit version of~\eqref{rderivexpfm} that we will prove is
\begin{eqnarray} 
\frac{d}{dp} \TFV (T , F, \V, p) & = & \frac{1}{p} \left [ 
   \sum_{e \in E(\V)} \TFV_{e} (T, F, \V, p) + 
   \sum_{v \in \partial \V} \TFV_{v+2} ( T, F, \V, p ) 
   \right. \nonumber \\
&& - \left. \left ( \sum_{e \in E(\V)} \TFV_{e+2} (T, F, \V , p ) 
   + \sum_{v \in \V \setminus \partial \V} 
     \TFV_{v+1} (T, F , \V , p ) \right )
   \right ] \label{eq:explicit}
\end{eqnarray}

\ul{Step 2: Convergence argument.}
The proof of~\eqref{eq:explicit} relies on the following fact
allowing us to interchange a derivative with an infinite sum.
If $B = \sum_{n=1}^\infty B_n$ converges everywhere on $(a,b)$
and $B_n$ are differentiable with derivatives $b_n$ such that 
$\sum_{n=1}^\infty b_n$ is continuous and $\sum_{n=1}^\infty |b_n (x)|$
converges to a function which is integrable 
on any interval $(r,s)$ for $a<r<s<b$, then $B$ is differentiable on $(a,b)$
and has derivative $\sum_{n=1}^\infty b_n$.  This follows, for example,
by applying the dominated convergence theorem to show that $\int_c^x 
\sum_{n=1}^M b_n (t) \, dt$ converges to $B(x) - B(c)$ for $a < c < x < b$.
The point of arguing this way is that it is good enough to have convergence 
of $\sum b_n$  to a continuous function and $\sum |b_n|$ 
to an integrable function; we do not need to keep demonstrating 
uniform convergence.

\ul{Step 3: Differentiating $\TFV |_{\Ttilde}$.}
We will apply this fact to~\eqref{eq:D sum} with $B = \TFV (T, F, \V, p)$,
and $\{ B_n \} = \{ \TFV |_{\Ttilde} \}$.  Each component function 
$\TFV |_{\Ttilde}$ is continuously differentiable on $(p_c,1)$ because 
$g$ and $g_2$ themselves are.  Differentiating $\TFV |_{\Ttilde}$ via the
product rule yields three terms, call them $b_{\Ttilde , j}$ for 
$1 \leq j \leq 3$:  
\begin{eqnarray*}
\frac{d}{dp} \TFV |_{\Ttilde}
& = & b_{\Ttilde , 1} + b_{\Ttilde , 2} + b_{\Ttilde , 3} \\
& = & \left [ \frac{d}{dp} \P_T(\open) \right ] 
   \P_T (\Nobranch) \P_T (\leafbranch) F(\V, \Ttilde) \\
&& + \P_T(\open) \left [ \frac{d}{dp} \P_T (\Nobranch)  \right ]
   \P_T (\leafbranch) F(\V, \Ttilde) \\
&& + \P_T(\open) \P_T (\Nobranch) \left [ \frac{d}{dp} \P_T (\leafbranch) 
   \right ] F(\V, \Ttilde) \, .
\end{eqnarray*}

The three sums $b_j := \sum_{\Ttilde \in \subt (T,\V)} b_{\Ttilde , j}$
will be identified as the three terms on the right-hand side 
of~\eqref{eq:explicit}, each of which is known, by Lemma~\ref{lem:finexpmon},
to be finite and continuous on $(p_c,1)$.  All the terms in
$\sum_{\Ttilde} b_{\Ttilde , j}$ have the same sign, namely 
positive for $j=1, 2$ and negative for $j=3$.  Therefore their absolute 
values sum to an integrable function.  It follows 
that $\TFV$ is continuously differentiable on $(p_c,1)$ with the 
explicit derivative given by~\eqref{eq:explicit}.  It remains only to 
identify the sums $b_1$, $b_2$ and $b_3$.

\ul{Step 4: Identifying jump rates.}
Recall from Lemma~\ref{lem:p markov} the Markov chain $\{ T_p : 
1 \geq p \geq 0 \}$.  This chain enters the set $\{ T(\V) = \Ttilde) \}$, 
which is equal to $\open \cap \Nobranch \cap \leafbranch$, exactly
when the function $\Nobranch$, which is decreasing in $p$, jumps
from~0 to~1 and the chain is already in the set $\open \cap \leafbranch$.
The chain leaves the set $\{ T(\V) = \Ttilde \}$ exactly when 
it is in the set and either $\open$ or $\leafbranch$ jumps 
from~1 to~0.  The three terms in the product rule correspond
to these three possibilities.  Figure~\ref{fig:jump} shows an example
of this: $\V$ is the illustrated 5-edge tree; $T_p (\V)$ and $T_{p+h} (\V)$
are respectively $\Ttilde$ and $\Ttilde'$; the jump into $\Ttilde$
as $p$ decreases occurs because $\Nobranch$ at the parent of $u$ flips 
from~0 to~1, while the jump out of $\Ttilde'$ is attributed to $\leafbranch$
flipping from~1 to~0 at the right child of $u$; the edge weight on the 
edge between $u$ and its parent also changes.

\begin{figure}[!ht] \label{fig:jump}
\centering
\subfigure[at parameter $p$] {
\includegraphics[width=2.7in]{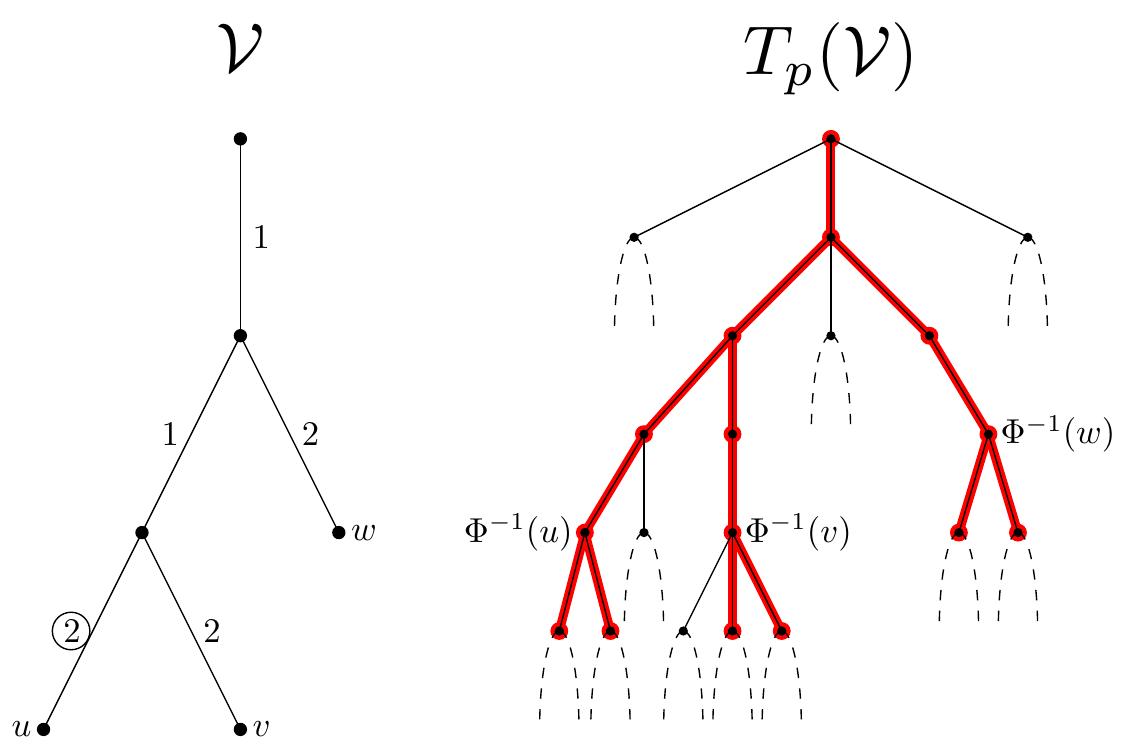}}
\subfigure[at parameter $p+h$] {
\hspace{0.5in}
\includegraphics[width=2.7in]{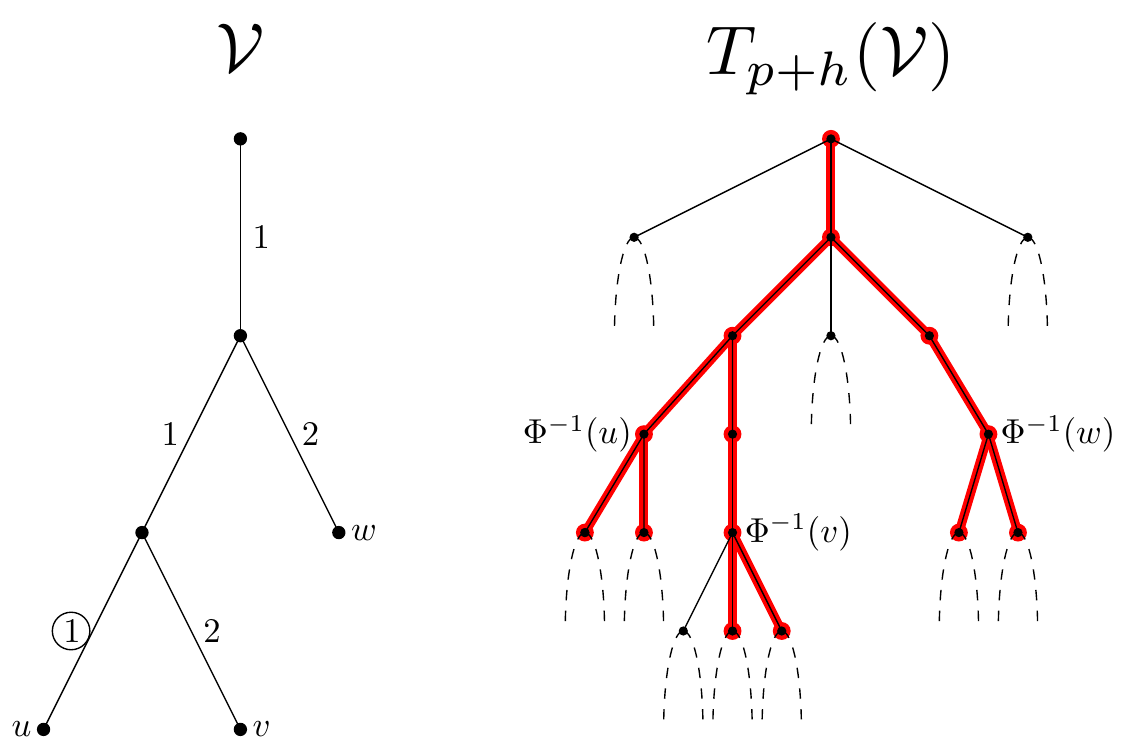}}
\caption{$T$ is in black while $\red T_p$ and $\red T_{p+h}$ 
	are in red.  The trees $T_p (\V)$ and $T_{p+h} (\V)$ are
	initial subtrees of the red tree.  Between times $p$ and $p+h$, some
	pivotal bond in the right subtree of $\Phi_{p+h}^{-1} (u)$ flips (this having previously been the right subtree of the parent of $\Phi_p^{-1}(u)$), causing 
	this portion of $T_{p+h}$ to branch earlier than $T_p$.  Further portions
	of $T$ are omitted in the figure of $T_{p+h}$ to stress that they 
	do not contribute to $T_{p+h}(\mathcal{V})$. }
\end{figure}

The jump rate out of $\open$, call it $\rho_1$, is the easiest to compute.  
Jumps out of $\open$ occur precisely when an edge in $\Ttilde$ closes,
which happens at rate $p^{-1} |E (\Ttilde)|$.  Writing $|E(\Ttilde)|$
as $\sum_{e \in E(\V)} d(e)$ shows that $\rho_1 F (\V, \Ttilde) = 
\sum_{e\in E(\V)}p^{-1} F_e(\V, \Ttilde)$.  Multiplying 
$\rho_1 F(\V , \Ttilde)$ by $\P(T_p(\V)=\Ttilde)$ and summing over 
$\Ttilde \in \subt (\V, T)$ gives $\sum_{e \in E(\V)} p^{-1} 
\TFV (T, F_e, \V, p)$, which is the first term in~\eqref{eq:explicit}.

Next we compute the jump rate out of $\leafbranch$.  
If such a jump takes place, let $q$ represent the time at which it occurs.  
Because $\leafbranch$ is an intersection over $v \in \partial \V$ of 
events that $\embed (v)$ has at least two children in $T_p$ (recall that 
$\embed$ refers to the embedding map defined near the beginning of this 
section), with probability~1 there is a unique $v_* \in \partial \V$ 
such that $\embed(v_*)$ has at least two children in $T_q$ and 
at most one child in $T_{q-}$, which means $\embed(v_*)$ has precisely two
children in $T_q$.  It follows, recalling the construction of $\TFV_{v+2}$
at the beginning of this proof, that $\V (v_*+2) \preceq T_q$ and
that closure of an edge $e \in T_q (\embed(v_*))$ causes an exit from 
$\leafbranch$ if and only if $e \in \embed (e_0) \cup \embed (e_1)$;
the edge must occur before the subtree $T_q (\embed(v_*))$ branches again.  
The rate $\rho_2$ is therefore equal to $\sum_{v_* \in \partial v}
p^{-1} (d(e_0) + d(e_1))$, where $e_j$ are the edges added to 
$\V$ in $\V(v_*+2)$.  Summing this last expression for $\rho_2$ over all 
the possible $\embed(e_0),\embed(e_1)$ combinations (multipled by their 
individual probabilities) and then multiplying by 
$F(\V , \Ttilde)\P(T_q(\V)=\Ttilde)$ and summing 
over $\Ttilde$ gives the second term in~\eqref{eq:explicit}.

Finally, we compute the jump rate into $\Nobranch$.  Such a jump 
can only occur in the following manner.  After the jump, $T_{q-} (\V) 
= \Ttilde$; at the time of the jump, $T_q (w)$ is infinite for 
precisely one child $w$ outside of $\Ttilde$ whose parent $v$ is an
interior vertex of $\Ttilde$; after the jump, one of the edges in
$T_q (w)$, either leading back to the parent or forward but 
somewhere before the first branch, closes.  

There are two possibilities.  First, the parent $v$ of $w$ might have 
at least two children in $T_q$ other than $w$.  In that case, $\V_i(\Phi(v)+1) 
\preceq T_q$ where the index $i$ refers to the relative placement of $w$ 
among the other children of $v$ (see definition $(i)$ at the beginning 
of the proof), and the edge that closes at time $q$ is in $\embed (e_*)$, 
where $e_*$ is the edge added to $\V$ to obtain $\V_i(\Phi(v)+1)$.
The rate at which this happens is $\rho_3 = p^{-1} d(e_*)$.
Summing over all possible $\embed(e_*)$ (multiplied by their probabilities), 
and then multiplying by $F(\V, \Ttilde)\P(T_q(\V)=\Ttilde)$ and summing over 
$i$ and then $\Ttilde$ gives $p^{-1}\TFV_{\Phi(v)+1} (T, F, \V, p)$.
The other possibility is that the parent of $w$ has only one other
child in $T_q$.  Let $e$ denote the descending edge from $\Phi(v)$.  
In this case, depending on whether $w$ is to the left or right of 
that child, $\V (e; L) \preceq T_q$ or $\V (e; R) \preceq T_q$.  
In either case, the jump rate is $\rho_4 := p^{-1} d(e_*)$, 
where as in construction~$(iii)$, the edge $e_* \in E(\V(e;L/R))$ 
corresponds to the path in $T_q$ containing the edge connecting $v$ to $w$.  
Now taking the sum of $\rho_4$ over all possible paths $\embed(e_*)$ 
(multiplied by their probabilities) and then multiplying by 
$F(\V, \Ttilde)\P(T_q(\V)=\Ttilde)$ and summing over the left and right 
cases and then over all $\Ttilde\in\subt$, we obtain
$p^{-1} \TFV_{e+2} (T, F, \V, p)$.  Next taking the sum over all 
$v\in\V\setminus\partial\V$ for the former case (where $v$ had at 
least two children other than $w$) and summing over all $e\in E(\V)$ 
for the latter, we find that the jump rate into $\Nobranch$ is indeed equal to
$$p^{-1} \left ( 
   \sum_{e \in E(\V)} \TFV_{e+2} (T, F, \V, p) + 
   \sum_{v \in \V \setminus \partial \V} \TFV_{v+1} (T, F, \V, p) \right )$$
which is the subtracted term in~\eqref{eq:explicit}.
$\Cox$

The next to last step in proving Theorem~\ref{th:cinf} is to 
apply Proposition~\ref{pr:JRformofderiv} to repeated derivatives
to obtain the following representation of the higher order derivatives
of $g(T,p)$ on $(p_c,1)$.

\begin{lem} \label{lem:iterated}
Under the same hypotheses, for any $k \geq 1$, there exists a finite
set of monomials of degree at most $k + \deg (F)$, call them 
$\{ F_\alpha : \alpha \in \A \}$, constants 
$\{C_{\alpha} : \alpha \in \A \}$, and corresponding collapsed trees 
$\{ \V_\alpha : \alpha \in \A \}$ of size at most $2k + |\E(\V)|$,
such that for $\GW$-almost every tree $T$,
\begin{equation} \label{eq:result 0}
\left ( \frac{d}{dp} \right )^k \TFV (T, F, \V,p) = \sum_{\alpha \in \A}
C_{\alpha} p^{-k-1 + \deg (F_\alpha) - \deg (F)} 
\TFV (T , F_\alpha , \V_\alpha , p) 
\end{equation}
on $(p_c , 1)$.  
\end{lem}

\noindent{\sc Proof:} Differentiate~\eqref{rderivexpfm} a total of
$k-1$ more times, using Proposition~\ref{pr:JRformofderiv} to
simplify each time.  Each time a derivative is taken, either it 
turns $p^{-j}$ into $-j p^{-j-1}$ for some $j$, or else 
a term $\TFV (T, F' , \V' , p)$ is replaced by a sum of terms
$\TFV (T, F'', \V'', p)$ where $\deg (F'') = 1 + \deg (F')$ and
$|E(\V'')| \leq 2 + |E(\V')|$.  The lemma follows by induction.  
$\Cox$

\medskip

\noindent{\sc Proof of Theorem \ref{th:cinf}}:
Applying Lemma~\ref{lem:iterated} to~\eqref{eq:N=1}, we see that
there are a pair of finite index sets $A$ and $B$, constants 
$\{C_{\alpha} : \alpha \in A \}$ and $\{C_{\beta} : \beta \in B \}$, 
collapsed trees $\{ \V_\alpha : 
\alpha \in A \}$ and $\{ \V_{\beta} : 
\beta \in B \}$, and corresponding monomials $\{ F_{\alpha} : 
\alpha \in A \}$ and $\{ F_{\beta} : 
\beta \in B \}$, such that 
\begin{align*}\left ( \frac{d}{dp} \right )^{k+1} g(T,p) = 
\left ( \frac{d}{dp} \right )^k p^{-1}\TFV (T, F_1 , \V_1 , p) &=
\sum_\alpha C_{\alpha} p^{-k-2 + \deg F_\alpha} 
\TFV (T, F_\alpha , \V_\alpha , p) \\ &+\sum_\beta C_{\beta} 
p^{-k-3 + \deg F_{\beta}} \TFV (T, F_{\beta} , \V_{\beta} , p)
\end{align*}
(note the need for the pair of distinct sums is on account of the 
$p^{-1}$ term in front of $\TFV (T, F_1 , \V_1 , p)$).  It follows 
from this that the $k$th derivative from the right of $g(T,p)$ 
exists for $\GW$-almost every tree $T$ and is given by an expression 
which is continuous in $p$.   
$\Cox$

\subsection{Continuity of the derivatives at $p_c$} \label{sec:cont-at-pc}

We now address the part of Theorem~\ref{th:main} concerning the 
behavior of the derivatives of $g$ near criticality.  We restate 
this result here as the following Theorem.

\begin{thm}\label{th:contderivpc}
If $\E[Z^{(2k^2 +3)(1+\beta)}]<\infty$ for some $\beta > 0$, then 
$$\lim_{p \to p_c^+} g^{(j)}(\TT,p)=j! M^{(j)}$$ 
for every $j\leq k\ \GW\text{-a.s.}$ where $M^{(j)}$ are as in 
Theorem~\ref{th:g-expansion}.
\end{thm}

To prove Theorem~\ref{th:contderivpc} we need to bound how badly
the monomial expectations $\TFV (T, F, \V, p_c + \ee)$ can blow
up as $\ee \downarrow 0$, then use Lemma~\ref{lem:g-upper-bound}
to see that they can't blow up at all.  

\begin{pr} \label{pr:k+1/k}
Let $\V$ be a collapsed tree with $\ell$ leaves and $\mathcal{E}$ 
edges and let $F$ be a monomial in the edges of $\V$.  Suppose that 
the offspring distribution has at least $m$ moments, where 
$m \geq \max_e F(e)$ and also $m \geq 3$.  Then
$$\TFV (T, F, \V, p_c + \ee ) = O \left ( \ee^\lambda \right )$$
for any $\lambda < 2\ell - \mathcal{E} - \deg (F)$ and $\GW$-almost 
every $T$.
\end{pr}

To prove this,
we first record the asymptotic behavior of the following annealed 
quantities at $p_c + \ee$, where $K = 2/(p_c^3 \phi''(1))$ as in 
Proposition~\ref{pr:K}.
\begin{lem} \label{lem:annealed}
Assume $\phi$ has at least three moments.  Then as $\ee\to 0^+$
\begin{eqnarray*}
g(p_c + \ee) & \sim & K \ee \\
1 - A_p & \sim & \mu \ee \\
g_2 (p_c + \ee) & \sim & K \mu \ee^2\,.
\end{eqnarray*}
\end{lem}

\noindent{\sc Proof:}
The first of these is Proposition~\ref{pr:K}.  For the second,
we recall from the definition that $\disp A_p = \left. \frac{d}{dz} 
\phi_p' (z) \right |_{z=0}$ and differentiate~\eqref{eq:phi_p} at 
$z=0$ to obtain
$$A_p = p \phi' (1 - p g(p)) \, . $$
Differentiating with respect to $p$ at $p = p_c$ and using $g(p_c) = 0$,
$\phi'(1) = \mu$ and $g'(p_c) = K$ then gives
\begin{eqnarray*}
\left. \frac{d}{dp} A_p \right |_{p_c} & = & \phi' (1 - p_c g(p_c))
   - [g(p_c) + p_c g'(p_c)] p_c \phi'' (1 - p_c g(p_c)) \\
& = & \phi' (1) - p_c^2 K \phi''(1) \\
& = & - \mu
\end{eqnarray*}
which proves the second estimate.  The third we can obtain by first 
calculating $\P(Z_p\geq 2)$ (recall $Z_p$ represents the size of the first 
generation of $\TT_p$ conditioned on $H(p)$), which is simply equal 
to $1-A_p$.  Now multiplying by $g(p)$ and using the first two estimates 
for $1-A_p$ and $g(p)$ near $p_c$, we obtain the third estimate.
$\Cox$

\noindent{\sc Proof of Proposition}~\ref{pr:k+1/k}:
We begin by computing annealed expectations.  Observe that 
$$\ell = 1 + \sum_{v \in \V \setminus \partial \V} (\deg_v - 1)$$
where $\deg_v$ is the number of children of $v$. 
Observe next that 
$$\P (Z_p \geq k) \leq \frac{1}{g(p_c + \ee)} 
   \E {Z \choose k} [(p_c+\ee) g(p_c + \ee)]^k 
   = O(\ee^{k-1})$$
under the assumption of at least $m$ moments with $k\leq m$ (where 
$Z$ represents the unconditioned offspring r.v.).  If $k>m$ then 
the above equality still gives us $\P(Z_p\geq k)\leq\P(Z_p\geq m)=O(\ee^{m-1})$.  
Putting these together we now obtain the expression 
\begin{equation} \label{eq:deg}\P (Z_p \geq k) = O(\ee^{(k\wedge m)-1})
\end{equation}whenever we have at least $m$ moments.  For future use, 
we also note here that if we condition on $Z_p\geq 2$ in the above 
inequalities, then the denominator in the expression above \eqref{eq:deg} 
becomes $g_2(p_c+\ee)$, and the expression on the right in \eqref{eq:deg} 
becomes $O(\ee^{(k\wedge m)-2})$.  Finally, recall as well that the 
collapsed annealed tree $\Phi(\TT_{p_c + \ee})$ has edge weights that are 
IID geometric random variables with means 
$1 / (1 - A_{p_c + \ee}) \sim \mu^{-1} \ee^{-1}$.  

Using these, we estimate the annealed expectation of $\TFV (\TT, F, 
\V, p_c + \ee)$ by using the branching process description of 
$\Phi(\TT_{p_c + \ee})$ to write the quantity 
$\E \TFV (\TT, F, \V, p_c + \ee)$ as the probability of the event 
$H(p)$ that $\TT_p$ is nonempty, multiplied by
the product of the following expectations which, once we condition on $H(p)$, 
are jointly independent.  
Assume that $Z$ has at least $m$ moments for $m = \max_e F(e)$.
The event $H(p_c + \ee)$ has probability 
$g(p_c + \ee) = O(\ee)$ for $\TT_p$.  Conditional on $H(p_c+\ee)$, there is 
a factor of $\E G^{F(e)} = O(\ee^{-F(e)})$ for each edge.  Lastly, because 
the non-root vertices of $\Phi(\TT_{p_c + \ee})$ consist precisely
of those non-root vertices of $\TT_{p_c + \ee}$ that posses more than 
one child, this amounts to conditioning on having at least two offspring.  
Therefore, there is a factor of $O(\ee^{\deg_v-1})$ for $v$ equal to the 
root, and a factor of $O(\ee^{\deg_v-2})$ for every other interior vertex 
of $\V$.  Multiplying all of these gives 
$$\E \TFV (\TT, F, \V , p_c + \ee) = O 
\left ( \ee^{2\ell - \mathcal{E} - \deg F} \right )$$
(note the exponent is simply a different way of writing 
$\ell-|\V\setminus\partial\V|+1-\deg F$).

The intuition behind the rest of the proof is as follows.
By Markov's inequality, the quenched expectation cannot be more than
$\ee^{-\delta}$ times this except with probability $\ee^\delta$.
As $\ee$ runs over powers of some $r<1$, these are summable, hence
by Borel-Cantelli, this threshold is exceeded finitely often.  
This only proves the estimate along the sequence $p_c + r^n$.
To complete the argument, one needs to make sure the quenched 
expectation does not blow up between powers of $r$.  
This is done by replacing the monomial expectation with an expression
$\Psi (T, F, \V, \ee_1 , \ee_2)$ that is an upper bound for the quenched 
expectation $\TFV (T, F, \V, x)$ as $x$ varies over an interval 
$[p_c + \ee_1 , p_c + \ee_2]$.

Accordingly, fix $\ee_1 < \ee_2$ and let $x$ vary over the interval 
$[p_c + \ee_1 , p_c + \ee_2]$.  Extend the notation in an obvious
manner, letting $\P_T (\open (\Ttilde , x))$ denote the probability
of all edges of $\Ttilde$ being open at parameter $x$, and similarly
for $\Nobranch$ and $\leafbranch$.  Fixing $T$, by monotonicity, 
\begin{eqnarray*}
\P_T (T_x (\V) = \Ttilde) & = & 
   \P_T (\open (\Ttilde , x)) 
   \P_T (\Nobranch (\Ttilde , x)) 
   \P_T (\leafbranch (\Ttilde , x)) \\
& \leq & 
   \P_T (\open (\Ttilde , \ee_2)) 
   \P_T (\Nobranch (\Ttilde , \ee_1)) 
   \P_T (\leafbranch (\Ttilde , \ee_2)) \, .
\end{eqnarray*}
Thus we may define the upper bound $\Psi$ by
$$\Psi (T, F, \V, \ee_1 , \ee_2) = \sum_{\Ttilde \in \subt (T , \V)}
   \P_T (\open (\Ttilde , \ee_2)) 
   \P_T (\Nobranch (\Ttilde , \ee_1)) 
   \P_T (\leafbranch (\Ttilde , \ee_2)) F(\Ttilde , \V) \, .$$

Taking the expectation, 
$$\E \Psi (\TT , F, \V, \ee_1 , \ee_2) = \int \left [ 
   \sum_{\Ttilde \in \subt (T , \V)}
   \P_{T} (\open (\Ttilde , \ee_2)) 
   \P_{T} (\Nobranch (\Ttilde , \ee_1)) 
   \P_{T} (\leafbranch (\Ttilde , \ee_2)) 
   \right ] \,  F(\Ttilde , \V) \, d\GW (T) .$$
In each summand, first integrate over the variables $\deg_w$
for $w \geq v$ with $v \in \partial \Ttilde$.  This replaces
$\P_T (\leafbranch (\Ttilde , \ee_2))$ by 
$g_2 (p_c + \ee_2)^{|\partial \Ttilde|}$.  Similarly, integrating
$\P_T (\leafbranch (\Ttilde , \ee_1)$ over just these variables 
would replace this with a factor of 
$g_2 (p_c + \ee_1)^{|\partial \Ttilde|}$.  Therefore, using that 
$|\partial\Ttilde|=|\partial\V|$ and noticing also that 
$\P_T (\open (\Ttilde , \ee_2)) = 
([p_c + \ee_2) / (p_c + \ee_1)]^{|E(\Ttilde)|}\P_T 
(\open (\Ttilde , \ee_1)) $, we see that
\begin{eqnarray*}
\E \Psi (\TT, F, \V , \ee_1 , \ee_2) & = &  
   \int \left [ 
   \sum_{\Ttilde \in \subt (T , \V)}
   \left ( \frac{g_2 (p_c + \ee_2)}{g_2 (p_c + \ee_1)} 
      \right )^{|\partial\V|} 
   \left ( \frac{p_c + \ee_2}{p_c + \ee_1} \right )^{|E (\Ttilde)|} 
   \P_{T} (\open (\Ttilde , \ee_1)) 
   \right. \\
&& \left. \cdot \, \P_{T} (\Nobranch (\Ttilde , \ee_1)) \cdot \, 
   \P_{T} (\leafbranch (\Ttilde , \ee_1))^{\phantom{~}}_{\phantom{~}} 
   \right ] \,  F(\Ttilde , \V) \, d\GW (T)  \\[2ex]
& = & 
   \int \left [ 
   \sum_{\Ttilde \in \subt (T , \V)}
   \left ( \frac{g_2 (p_c + \ee_2)}{g_2 (p_c + \ee_1)} 
      \right )^{|\partial\V|} 
   \left ( \frac{p_c + \ee_2}{p_c + \ee_1} \right )^{|E (\Ttilde)|} 
   \P_T (T_{p_c + \ee_1} (\V) = \Ttilde) 
   \right ] \,  F(\Ttilde , \V) \, d\GW (T). 
\end{eqnarray*}
To integrate over $d\GW (T)$, recall that the edge weights $d(e)$ 
in $\Phi(T_{p_c + \ee_1})$ are IID geometrics with mean 
$1/(1 - A_{p_c + \ee_1})$ and independent from the degrees 
$\deg_{\iota (v)}$.  Let $C$ denote any upper bound for $\P (H(p))$ 
times the product over interior vertices $v\in\V$ of the quantity 
$\ee_1^{\deg_v - j} \P_p (Z_p = \deg_v|Z_p\geq j)$, which is finite 
by~\eqref{eq:deg} (note we're assuming here that $j=1$ for the root, 
which means no conditioning, and $2$ for all other interior vertices).  
Let $G$ denote such a geometric random variable as referenced above, 
let $\alpha$ denote an upper bound for 
$g_2 (p_c + \ee_2) / g_2 (p_c + \ee_1)$ and 
$1 + \delta$ denote an upper bound for $(p_c + \ee_2) / (p_c + \ee_1)$.
Integrating against $d\GW (T)$ now yields
$$\E \Psi (\TT, F, \V , \ee_1 , \ee_2) \leq
   \alpha^{|\partial \V|} \cdot 
   C \ee_1^{2\ell-\mathcal{E}} \cdot 
   \prod_{e \in E(\V)} \left [ \E (1 + \delta)^G G^{F(e)} \right ] 
   \; .$$

Because $\P (G=k) = (1-A_p) A_p^{k-1}$, we see that 
for any $f$, $(1+\delta)^G f(G)$ may be computed as 
\begin{align*}\E (1+\delta)^G f(G) = \sum_{k=1}^\infty 
(1+\delta)^k (1-A_p) A_p^{k-1} f(k) 
&= \frac{(1-A_p)(1+\delta)}{1 - (1+\delta) A_p}
\sum_{k=1}^\infty (1 - (1+\delta) A_p) ((1+\delta) A_p)^{k-1} f(k) \\ 
&= \frac{(1-A_p)(1+\delta)}{1 - (1+\delta) A_p} \E f(G')\end{align*}
where $G'$ is a geometric r.v. with parameter $(1+\delta) A_p$.  Thus,
\begin{equation} \label{eq:Psi UB}
\E \Psi (\TT, F, \V , \ee_1 , \ee_2) \leq C \alpha^{|\partial \V|}
\ee_1^{2\ell-\mathcal{E}}\frac{(1-A_p)(1+\delta)}{1 - (1+\delta) A_p} 
\prod_{e \in E(\V)} \E (G')^{F(e)} \, .
\end{equation} 

If we choose $\ee_2 < \frac{3}{2}\ee_1$ then 
$\delta = (p_c + \ee_2)/(p_c + \ee_1) - 1 < \frac{\mu}{2}\ee_1$ 
which implies $1 - (1+\delta) A_p > (1/2) (1 - A_p)$ as $\ee_1\to 0$.  
This in turn implies that as $\ee_1\to 0$ we have 
$\E (G')^k \leq 2^k \E G^k = O(\ee^{-k})$, and hence

\begin{equation} \label{eq:E Psi}
\E \Psi (\TT, F, \V , \ee_1 , \ee_2) = O 
\left ( \ee^{2\ell - \mathcal{E} - \deg (F)}\right ) \, .
\end{equation}

Now the Borel-Cantelli argument is all set up.  Let 
$\ee_n := (\frac{4}{3})^{-n}$ and apply the previous argument
with $\ee_2 = \ee_n$ and $\ee_1 = \ee_{n+1}$.  We see that
$$\E \Psi (\TT, F, \V , \ee_{n+1} , \ee_n) 
= O \left ( \ee_{n+1}^{2\ell - \mathcal{E} - \deg (F)}\right )$$
and hence by Markov's inequality,
$$\P \Big( \Psi (\TT, F, \V , \ee_{n+1} , \ee_n)>\ee_{n+1}^{2\ell - 
\mathcal{E} - \deg (F) - t}\Big) = 
O\Big(\Big(\frac{4}{3}\Big)^{-tn}\Big) \, .$$
This is summable, implying these events occur finitely often,
implying that the quenched survival function satisfies 
$$\TFV (T, F, \V, p_c + \ee) = O \left ( 
   \ee^{2\ell - \mathcal{E} - \deg (F) - t} \right )$$
for $\GW$-almost every $T$.
$\Cox$

\medskip
Before presenting the proof of Theorem \ref{th:contderivpc}, 
we just need to establish one final lemma.

\medskip
\begin{lem} \label{lem:bvsedgs} Every non-empty collapsed tree 
$\V$ satisfies the inequality $2\ell-\mathcal{E}\geq 1$.\end{lem}

\begin{proof}
Certainly the lemma holds for any $\V$ of height one.  Now if $\V$ 
represents any non-empty collapsed tree for which the lemma applies 
and we add $k$ children to one of the boundary vertices of $\V$
(note we can assume $k\geq 2$ since all non-root interior vertices 
of a collapsed tree must have at least $2$ children), then the value 
of $2\ell-\mathcal{E}$ is increased by $2(k-1)-k = k-2 \geq 0$, which 
means the lemma still applies to the new collapsed tree.  Since any 
collapsed tree can be obtained from a height one tree via finitely 
many of these steps, the lemma then follows.$\Cox$
\end{proof}

\bigskip
{\sc Proof of Theorem~\ref{th:contderivpc}:}
Presume the result holds for all $j<k$.  Lemma~\ref{lem:iterated} 
expresses $(d/dp)^{k+1} g(T,p)$ on $(p_c,1)$ 
as a sum of terms of the form $\TFV(T, F_\alpha, \V_\alpha , p)$ with
$\deg (F_\alpha) \leq k+1$ and $|E(\V_\alpha)| \leq 2k+1$.
By Proposition~\ref{pr:k+1/k}, because our moment assumption implies
the weaker moment assumption of $k+1$ moments, each summand on the 
right-hand side of~\eqref{eq:result 0} is 
$O(\ee^{2\ell - \mathcal{E} - \deg (F)})$.
The worst case is $\deg (F) = k+1$ and 
$2\ell - \mathcal{E} = 1$ (see lemma above).  Therefore, 
the whole sum satisfies
\begin{equation} \label{eq:big O}
\left | g^{(k+1)} (T,p) \right | = O \left ( \ee^{-k - t} \right )
\end{equation}
for any $t > 0$ and $\GW$-almost every $T$.

Suppose now for contradiction that $g^{(k)} (\TT , p_c + \ee)$
does not coverge to $k! \, c_k$ as $\ee \downarrow 0$.  We apply
Lemma~\ref{lem:g-upper-bound} with $N = k^2 + 1$.  To check the
hypotheses, note that the existence of the order-$N$ expansion
follows from Theorem~\ref{th:g-expansion} with $\ell = k^2 + 1$
and our assumption that $\E[Z^{(1+\beta)(2k^2+3)}] < \infty$.
The induction hypothesis implies hypothesis~\eqref{eq:j<k} of
Lemma~\ref{lem:g-upper-bound}.  Our
assumption for proof by contradiction completes the verification
of the hypotheses of Lemma~\ref{lem:g-upper-bound}.  Seeing that
$N/k = k + 1/k$, the conclusion of the lemma directly 
contradicts~\eqref{eq:big O} when $t < 1/k$.  Since the proof of 
the induction step also establishes the base case $k=1$, the proof 
of the theorem is complete.
$\Cox$

\section{Open questions}

We conclude with a couple of open questions.
Propositions~\ref{pr:nonsmpr} and~\ref{pr:simfannc} are converses 
of a sort but they leave a gap as to whether $g \in C^j$ from the 
right at $p_c^+$ for $k/2 \leq j \leq k$.  

\begin{question}
Do $k$ moments of the offspring distribution suffice to imply that 
the annealed survival function is $k$ times differentiable at $p_c^+$?
More generally, is there a sharp result that $k$ moments but not $k+1$ 
imply $j$ times differentiability but not $j+1$ for some $j \in [ 
\lfloor k/2 \rfloor , k]$?
\end{question}

Another question is whether the expansion of either the quenched or
annealed survival function at $p_c$ has any more terms than are 
guaranteed by the continuity class.

\begin{question}
Does it ever occur that $g$ has an order-$k$ expansion at $p_c^+$ 
but is not of class $C^k$ at $p_c^+$?  Does this happen with positive
probability for $g(\TT , \cdot)$?
\end{question}

Recall that the annealed survival function is analytic on $(p_c , 1)$
whenever the offspring generating function extends beyond~1.  We do not 
know whether the same is true of the quenched survival function.

\begin{question}
When the offspring distribution has exponential moments, is the 
quenched survival function $g(\TT , p)$ almost surely analytic
on $[p_c , 1)$?
\end{question}

\section*{Acknowledgments}
The authors would like to thank Michael Damron for drawing our 
attention to~\cite{F-vdH}.  We would also like to thank Yuval Peres
for helpful conversations.

	\bibliographystyle{alpha}
	\bibliography{Bib}
	
\end{document}